\documentclass[a4paper, 11pt]{amsart}

\usepackage{a4wide}
\usepackage[english]{babel} 
\usepackage{amsmath, amssymb} 
\usepackage{tikz}
\usepackage{tikz-cd}
	\tikzcdset{row sep/normal=1.3em}
	\tikzcdset{column sep/normal=1.3em}
\usepackage[normalem]{ulem}
\usepackage[mathscr]{euscript}

\usepackage{enumerate}
\usepackage{hyperref}
\usepackage[capitalize,noabbrev]{cleveref}

\newcommand{\fun}{\mathrm{Fun}}
\renewcommand{\hom}{\mathrm{hom}}

\newcommand{\Q}{\mathbb{Q}}

\newcommand{\Z}{\mathbb{Z}}
\newcommand{\F}{\mathbb{F}}

\newcommand{\bbS}{\mathbb{S}}

\newcommand{\M}{\mathrm{M}}

\newcommand{\K}{K_{\mathrm{nc}}}

\newcommand{\cO}{\mathcal{O}}
\newcommand{\cA}{\mathcal{A}}

\newcommand{\cC}{\mathcal{C}}
\newcommand{\cM}{\mathcal{M}}

\newcommand{\cD}{\mathcal{D}}

\newcommand{\cP}{\mathcal{P}}

\newcommand{\E}{\mathbb{E}}

\newcommand{\swedge}{{\scriptscriptstyle\wedge}}
\newcommand{\tensor}{\otimes}

\DeclareMathOperator{\hgt}{ht}

\newcommand{\op}{\mathrm{op}}
\newcommand{\Cat}{\mathrm{Cat}}

\newcommand{\exact}{\mathrm{ex}}
\newcommand{\perf}{\mathrm{perf}}

\renewcommand{\inf}{\mathrm{inf}}
\newcommand{\GL}{\mathrm{GL}}
\newcommand{\BGL}{\mathrm{BGL}}

\newcommand{\ku}{ku}
\newcommand{\ko}{ko}

\newcommand{\tmf}{tm\mkern-2mu f}
\newcommand{\Tmf}{T\mkern-2mu m\mkern-2mu f}

\newcommand{\adj}{[\tfrac{1}{p}]}

\newcommand{\B}{\mathrm{B}}
\newcommand{\add}{\mathrm{add}}
\newcommand{\Stab}{\mathrm{Stab}}
\newcommand{\Ab}{\mathrm{Ab}}

\newcommand{\lto}{\longrightarrow}

\newcommand{\id}{\mathrm{id}}

\DeclareMathOperator{\Spec}{Spec}
\DeclareMathOperator{\Spc}{Spc}
\DeclareMathOperator{\Sp}{Sp}

\DeclareMathOperator*{\colim}{colim}
\newcommand{\laxtimes}[1]{\mathop{\times\mkern-13mu\raise1.3ex\hbox{$\scriptscriptstyle\to$}_{#1}}}
\newcommand{\slax}{ \times\mkern-14mu\raise1ex\hbox{$\scriptscriptstyle\to$} }

\DeclareMathOperator{\RMod}{RMod}
\DeclareMathOperator{\Mod}{Mod}

\DeclareMathOperator{\Perf}{Perf}

\DeclareMathOperator{\Map}{Map}
\DeclareMathOperator{\map}{map}
\DeclareMathOperator{\End}{End}
\DeclareMathOperator{\Fun}{Fun}

\DeclareMathOperator{\Ind}{Ind}
\DeclareMathOperator{\fib}{fib}

\DeclareMathOperator{\TC}{TC}

\DeclareMathOperator{\THH}{THH}

\newcommand{\wtimes}[2]{\odot_{#1}^{#2}}

\newtheorem{thm}{Theorem}
\newtheorem*{thm*}{Theorem}
\newtheorem*{puritythm}{Purity Theorem}
\newtheorem*{redshift}{Redshift Theorem}
\newtheorem{cor}[thm]{Corollary}
\newtheorem*{cor*}{Corollary}
\newtheorem{lemma}[thm]{Lemma}
\newtheorem{prop}[thm]{Proposition}

\newtheorem{quest}[thm]{Question}

\theoremstyle{definition}
\newtheorem{dfn}[thm]{Definition}
\newtheorem*{dfn*}{Definition}

\newtheorem*{ex*}{Example}
\newtheorem{ex}[thm]{Example}

\theoremstyle{remark}

\newtheorem*{claim*}{Claim}

\numberwithin{thm}{section}

\newtheorem{rem}[thm]{Remark}
\newtheorem*{rem*}{Remark}

\theoremstyle{plain}
\newcounter{zaehler}

\newtheorem{introthm}[zaehler]{Theorem}
\newtheorem{introcor}[zaehler]{Corollary}
\newtheorem*{introcor*}{Corollary}

\title{Purity in chromatically localized algebraic \textit{K}-theory}

\author{Markus Land}
\address{Mathematisches Institut, Ludwig-Maximilians-Universit\"at M\"unchen, Theresienstr. 39, 80333 M\"unchen, Germany}
\email{markus.land@math.lmu.de}

\author{Akhil Mathew}
\address{Department of Mathematics, University of Chicago, 5734 S University
Ave, Chicago, IL 60637 USA}
\email{amathew@math.uchicago.edu}

\author{Lennart Meier}
\address{Mathematical Institute, Utrecht University, Budapestlaan 6, 3584 CD Utrecht, The Netherlands}
\email{f.l.m.meier@uu.nl}

\author{Georg Tamme}
\address{Institut f\"ur Mathematik, Fachbereich 08, Johannes Gutenberg-Universit\"at Mainz, 55099 Mainz, Germany}
\email{georg.tamme@uni-mainz.de}

\thanks{The first and fourth authors were partially supported by the CRC/SFB 1085
\emph{Higher Invariants} (Universit\"at Regensburg) funded by the DFG. The first
author was further supported by the DFG through a research fellowship and by the
Danish National Research Foundation through the Centre for Symmetry and
Deformation (DNRF92) and the Centre for Geometry and Topology (DNRF151). Results
incorporated in this paper have received funding from the European Union's
Horizon 2020 research and innovation programme under the Marie Sklodowska-Curie
grant agreement No 888676. This work was done while the second author was
supported by a  Clay
Research Fellowship and by the National Science Foundation under Grant No.~2152311. The third author was supported by the NWO through VI.Vidi.193.111. The fourth author was further partially supported by the DFG through TRR 326 (Project-ID 444845124).}

\date{\today}

\begin{document}

\begin{abstract}
We prove a purity property in telescopically localized algebraic
$K$-theory of ring spectra: For $n\geq 1$, the $T(n)$-localization of $K(R)$ only depends on the $T(0)\oplus \dots \oplus T(n)$-localization of $R$. This complements a classical result of Waldhausen in rational $K$-theory. Combining our result with work of Clausen--Mathew--Naumann--Noel, one finds that $L_{T(n)}K(R)$ in fact only depends on the $T(n-1)\oplus T(n)$-localization of $R$, again for $n \geq 1$. As consequences, we deduce several vanishing results for telescopically localized $K$-theory, as well as an equivalence between $K(R)$ and $\TC(\tau_{\geq 0} R)$ after $T(n)$-localization for $n\geq 2$.
\end{abstract}

\maketitle

\tableofcontents

\section{Introduction}

The subject of this paper is the algebraic $K$-theory of (structured) ring
spectra. The field was initiated by Waldhausen \cite{Waldhausen}, with motivations from 
manifold topology, cf.~\cite{WaldhausenJahrenRognes} and has recently seen spectacular applications to the telescope conjecture in chromatic homotopy theory \cite{BHLS}.

Placing even classical discrete rings in the context of ring spectra (or derived rings) is very fruitful: For instance, derived blowups play an important role in the solution of Weibel's conjecture \cite{KST} and the approach to excision in $K$-theory of \cite{LT, LT2} also illustrates how ring spectra arise naturally even in
questions involving only the $K$-theory of classical rings. In addition, trace methods can be used to approximate $K$-theory 
by topological cyclic homology and have been a major computational technique for the past decades (see e.g.\ \cite{BokstedtHsiangMadsen,HesselholtMadsenFinite, DGM, CMM}). 

For a given ring spectrum $R$, it is often easier to study completions or localizations of the spectrum $K(R)$ than directly the homotopy groups of $K(R)$, i.e.\ the groups $K_i(R)$. For example, it is a theorem of Suslin \cite{Suslin1} that the $p$-adic completion of the algebraic $K$-theory of separably closed fields is fully understood, much unlike its rationalization. Conversely Borel \cite{Borel} computed the rational $K$-theory of number rings long before substantial information (beyond finite generation) about their integral behavior was known. Moreover, Waldhausen \cite{Waldhausen}
proved that a rational equivalence between connective ring spectra which is a
$\pi_0$-isomorphism induces an equivalence in rational algebraic $K$-theory. As the map $\bbS \to \Z$ from the sphere spectrum to the integers is a rational equivalence, this resulted in a computation of the rational $K$-theory of $\bbS$, which was of great use in geometric topology, see e.g.\ \cite{FarrellHsiang}. 

Our article is concerned with generalizations of Waldhausen's result in the context of chromatic homotopy theory, which organizes stable homotopy theory by a notion of height in which classical discrete rings have height $\leq 0$. From this viewpoint, inverting $p$, or equivalently localization at $\bbS[\tfrac{1}{p}] = T(0)$, for which the analog of Waldhausen's result is equally correct, is merely the zeroth in a whole sequence of localizations depending on a fixed prime $p$. The higher height analogs at prime $p$ are given by localization with respect to a height $n$ telescope $T(n)$, see \cref{recollections} for details.
The associated Bousfield localization functor $L_{T(n)}$ isolates phenomena at height $n$, just like $p$-completion isolates phenomena at the prime $p$. We refer to $L_{T(n)}$ as \emph{telescopic localization} (at height~$n$ and implicit prime $p$).

For $n=1$, the telescopic localization agrees with the better known localization at Morava $K$-theory $K(1)$ by \cite{Miller,Mahowald}. Moreover, in seminal work \cite{Thomason}, Thomason related \'etale algebraic $K$-theory to the ${K(1)}$-localization of $K$-theory thereby showing that algebraic $K$-theory interacts in an intriguing way with the chromatic filtration in low heights.  The Lichtenbaum--Quillen conjecture, now a theorem due to work of Voevodsky and Rost \cite{V1,V2}, can therefore be formulated as describing the height one information of the algebraic $K$-theory of rings of height $0$. Moreover, it was shown by Mitchell that the higher telescopic localizations of algebraic $K$-theory of classical rings vanish \cite{Mitchell}. This, together with concrete calculations with Ausoni \cite{AR}, led Rognes to formulate the redshift philosophy and higher chromatic versions of Quillen--Lichtenbaum type conjectures. Roughly speaking, these conjectures say that the algebraic $K$-theory of a ring spectrum of height $n$ should be of height $(n+1)$, in a quantitative manner. We explain more about redshift and how our results contribute to its understanding later in this introduction.

In the context of the interaction of algebraic $K$-theory with different chromatic heights, Bhatt--Clausen--Mathew \cite{BCM} have recently proven the following theorem, based on
arithmetic techniques, in particular the theory of prismatic cohomology and its relationship with topological cyclic homology, \cite{BMS2, Prisms}:
A map of $H\Z$-algebras which is a $T(0)$-equivalence induces an equivalence in $T(1)$-local $K$-theory. Taking all of the above into account, one is naturally led to the following question: 
\begin{quote}
For general height $n$, to what extent does the $T(n)$-local $K$-theory of a ring spectrum $A$  depend only on telescopic localizations of $A$ ?
\end{quote}
Except from the above results of Waldhausen and Bhatt--Clausen--Mathew nothing in this direction was known.
The main goal of this paper is to give a complete  answer to this question at all chromatic heights $n \geq 1$. The strongest form of our result --- the Purity Theorem below --- is proved jointly between this paper and the related work \cite{CMNN2}.
Our contribution here is the following theorem.

\begin{introthm} \label{thm-A} Let $A$ be a ring spectrum. 
	\begin{enumerate}
	\item For $n\geq 1$, the canonical map $A \to L_{T(0) \oplus \dots \oplus T(n)}A$ induces an equivalence on $T(n)$-local $K$-theory.
	\item For $n\geq 2$, the canonical map $A \to L_{T(1)\oplus \dots \oplus T(n)}A$ induces an equivalence on $T(n)$-local $K$-theory.
	\end{enumerate}
\end{introthm}

As a direct consequence one obtains for instance that, for $n\geq 2$, the $T(n)$-local $K$-theory of a ring spectrum depends only on its connective cover, allowing for effective use of trace methods. For instance, Clausen--Mathew--Naumann--Noel made use of it in establishing the following vanishing result, and subsequently Ben-Moshe--Carmeli--Schlank--Yanovski made use of such results in their work on hyperdescent for $K(n)$-local $K$-theory \cite{BMCSY}.

\begin{thm}[{\cite{CMNN2}}]\label{thm-cmnn}
For $n\geq 2$, the $K$-theory of $L_{T(0)\oplus \dots \oplus T(n-2)}\bbS$ vanishes $T(n)$-locally.
\end{thm}

In fact, Clausen--Mathew--Naumann--Noel prove a more general result, see \cite[Theorem~C]{CMNN2}. For the reader's convenience we indicate the proof of the above special case in  Remark~\ref{CMNN-argument}. Combining Theorems~\ref{thm-A} and \ref{thm-cmnn}, a fracture square argument yields the following result.

\begin{puritythm}
Let $A$ be a ring spectrum. For $n\geq 1$, the canonical map $A \to L_{T(n-1)\oplus T(n)}A$ induces an equivalence in $T(n)$-local $K$-theory.
\end{puritythm}

The purity theorem is optimal in the sense that the functor $A \mapsto L_{T(n)}
K(A)$ does not factor through either $A \mapsto L_{T(n-1)}A$ or $A \mapsto L_{T(n)}A$
(see \Cref{purity:optimal} below). 
In the following, we discuss some consequences of our results.

\subsection*{Redshift}

As indicated before, in seminal papers Thomason \cite{Thomason,TT} proved that (under mild finiteness assumptions)
the $T(1)$-local $K$-theory of algebraic spaces in which $p$ is invertible satisfies \'etale hyperdescent and describes their $T(1)$-local $K$-groups in terms of \'etale cohomology groups with coefficients in $p$-adic Tate twists; see \cite[Theorem~3.9]{BCM} for a spectrum level version. As
 a consequence, one can formulate the proven Lichtenbaum--Quillen conjecture as asserting that for schemes over $\Z[\frac{1}{p}]$ the $p$-adic $K$-theory is asymptotically $T(1)$-local, that is, it agrees with its $T(1)$-localization in high enough degrees (\cite[\S 4]{WaldhausenLoc}, see \cite{CM2} for a modern account and precise statements). 
Informally, one can think of $T(1)$-localization as enforcing a form of Bott periodicity so that the $p$-adic $K$-theory is Bott periodic in high degrees.

Ausoni--Rognes conjectured higher chromatic analogs of such statements and provided evidence through 
 precise computations of the mod $(p, v_{1})$-homotopy of the algebraic $K$-theory of  $\ell_{p}= BP\langle1\rangle$, $ku_{p}^{\swedge}$, and $\ell_{p}/p$ \cite{ARconj,AR,Ausoni, AR2}. 
Quite recently, \'etale hyperdescent for telescopically localized $K$-theory at any height has been established in \cite{CMNN,CM2}.
Exciting progress in the direction of the Lichtenbaum--Quillen conjectures at arbitrary height has been made by Hahn--Wilson \cite{HW21} verifying these for the truncated Brown--Peterson spectra $BP\langle n\rangle$ for all $n$.

These higher chromatic Lichtenbaum--Quillen conjectures suggest in particular that algebraic $K$-theory increases chromatic complexity by 1, a property which has been coined \emph{redshift}.     
Despite the computational advances mentioned above, a complete structural understanding of redshift has been missing so far. Now,
for $\E_{\infty}$-rings there is a particularly well behaved notion of chromatic complexity
called the height.
Namely Hahn \cite{Hahn} proved that, for any $n$, a  $T(n)$-acyclic $\E_{\infty}$-ring is also $T(n+1)$-acyclic, and one defines the height  of an $\E_{\infty}$-ring $R$ as 
\[
\hgt(R) = \inf \{  n \geq -1 \,|\, R \otimes T(n+1) = 0 \}.
\]
One precise formulation of redshift is that for $\E_{\infty}$-rings of non-negative height, one has $\hgt( K(R) ) = \hgt(R)+1$.\footnote{Since $\hgt(K(\F_p^{tC_p})) = -1 \neq \hgt(\F_p^{tC_p}) + 1$, one should indeed restrict this formulation of redshift to non-negative height rings (or introduce a finer notion of negative heights, so that $\hgt(\F_p^{tC_p}) = -2$).} Prior to our work, a general form of redshift was wide open. 
As an immediate consequence of the Purity Theorem one obtains the following inequality (see also \cite[Theorem~A]{CMNN2} for this joint result):

\begin{introthm}\label{introthmB}
For an $\E_{\infty}$-ring $R$, we have $\hgt(K(R)) \leq \hgt(R) + 1$.
\end{introthm}
The converse inequality is by now in fact also known, and follows from works of Yuan and Burklund--Schlank--Yuan each of which appeared after the first versions of the present article: Making use of Theorem~\ref{thm-A}, Yuan proved that $\hgt(K(E_n)) \geq \hgt(E_{n})+1$ where $E_n$ is the Lubin--Tate theory associated to a formal group $\mathbb{G}$ of height $n$ over a perfect field $k$ of characteristic $p$, see \cite[Theorem~A]{Yuan}. He uses our result, together with several other ingredients, to reduce the question to showing $\hgt(K((\tau_{\geq 0}E_n)^{tC_p}))\geq n$, for which he can employ an argument using the $\mathbb{T}$-equivariant Dennis trace map. Then, in their work on higher chromatic analogs of Hilbert's Nullstellensatz, Burklund--Schlank--Yuan \cite[Theorem D]{BSY} show that, for $n\geq 0$ and any $\E_{\infty}$-ring $R$ with $L_{T(n)}R\not=0$, there exists an $\E_{\infty}$-map $R \to E_{n}$ for  an appropriate choice of $(\mathbb G, k)$.\footnote{For $n=0$, $E_n$ is given by $k[t^{\pm}]$ with $|t|=2$ and $k$ of characteristic 0.} Combining these results with Theorem~\ref{introthmB}, one thus obtains redshift for $\E_\infty$-rings:
\begin{redshift}[\cite{Yuan, BSY}, Theorem~\ref{introthmB}]\hypertarget{thm:Redshift}
For an $\E_\infty$-ring $R$ with $\hgt(R) \geq 0$, we have  $\hgt(K(R)) = \hgt(R)+1$.
\end{redshift}

One can interpret this result as follows. We recall that associated to any spectrum $X$ is its telescopic filtration $\{ L_n^{p,f}X \}_{n \geq 0}$ which is closely related to the chromatic filtration $\{ L_n^{p}X\}_{n \geq 0}$ alluded to in the very beginning of this introduction, see Section~\ref{recollections} for further background. It follows from Lemma~\ref{telescope conjecture on ring spectra} that for a ring spectrum, the $n$th graded piece in the telescopic filtration vanishes if and only if the $n$th graded piece in the chromatic filtration vanishes.
The \hyperlink{thm:Redshift}{Redshift Theorem} can then be interpreted as stating the following. For an $\E_\infty$-ring $R$ with $\hgt(R)\geq 0$, the following assertions are equivalent:
\begin{enumerate}
\item[$\bullet$] The $n$th graded piece of the chromatic filtration of $R$ is trivial.
\item[$\bullet$] The $(n+1)$th graded piece of the chromatic filtration of $K(R)$ is trivial.
\end{enumerate}

\subsection*{Further applications}

Many results follow quite quickly from the Purity Theorem. We list a few of them here, and refer to the body of the text for more applications and explanations.

\begin{introcor}
For a $T(1)$-acyclic ring spectrum $A$, the canonical map $K(A) \to K(A[\frac{1}{p}])$ is a $T(1)$-local equivalence.
\end{introcor}
This recovers the result of Bhatt--Clausen--Mathew \cite{BCM} for $H\Z$-algebras mentioned earlier and gives a new proof by purely homotopy and $K$-theoretic methods.
As a consequence, $T(1)$-local $K$-theory is truncating on $T(1)$-acyclic ring spectra and therefore satisfies excision, nilinvariance and cdh-descent. In addition, it is $\mathbb{A}^{1}$-homotopy invariant and thus identifies with the $T(1)$-localization of Weibel's homotopy $K$-theory,\footnote{From this fact, excision, nilinvariance, and cdh-descent also follow.} see Section~\ref{subsec:K1-local-consequences}.

As further sample applications of the Purity Theorem, we offer the following corollary, special cases of which appear in work of 
Ausoni--Rognes, Angelini-Knoll--Salch, and Hahn--Wilson \cite{AR, AR2, AKS,HW21}.

\begin{introcor}
	\begin{enumerate}
	\item For positive integers $m\neq n,n+1$, the spectrum $K(K(n))$ is $T(m)$-acyclic.
	\item The map $K(BP\langle n \rangle) \to K(E(n))$ is a $T(m)$-equivalence for $m \geq n+1$, and both terms vanish $T(m)$-locally for $m\geq n+2$.
	\end{enumerate}
\end{introcor}

See Section~\ref{subsec:examples1} for details, further results, and a discussion of  (ii) in light of a question of Rognes.
The following two corollaries are discussed in Section~\ref{subsec:KTC}.

\begin{introcor}
Let $A$ be a ring spectrum. For $n\geq 2$, the map induced by the connective cover and the cyclotomic trace
\[
K(A) \longleftarrow K(\tau_{\geq 0}A) \lto \TC(\tau_{\geq 0}A)
\]
are $T(n)$-local equivalences.
\end{introcor}
Together with results of \cite{LRRV, CMM}, this gives the following Farrell--Jones type equivalence. 
We denote by $\mathscr{O}_{\mathscr{C}}(G)$ the orbit category of a group $G$ with respect to the family of cyclic subgroups, that is, the full subcategory of the category of $G$-spaces on transitive $G$-sets of the form $G/H$ with $H \subseteq G$ a cyclic subgroup. For a ring spectrum $A$ and any group $G$, we write $AG = A \otimes \Sigma^\infty_+ G$ for the group ring of $G$ with coefficients in $A$.
\begin{introcor}
For a ring spectrum $A$, a group $G$, and $n\geq 2$, the assembly map in algebraic $K$-theory for the family of cyclic subgroups
	\[
	\colim\limits_{G/H \in \mathscr{O}_\mathscr{C}(G)} K(AH) \lto K(AG) 
	\]
	is a $T(n)$-local equivalence.
\end{introcor}

\bigskip

\noindent \textsc{Conventions.}
We fix a prime number $p$ which will be the implicit prime in all Morava $K$-theories $K(i)$ and $T(i)$ below.
We adopt the convention that $K(0) = H\Q$.
Whenever we speak of a ring spectrum, we mean an $\E_1$-ring spectrum, i.e.~an algebra in the symmetric monoidal $\infty$-category $\Sp$ of spectra. By a  module over a ring spectrum we mean  a right module.
Given an $\E_k$-ring spectrum $R$ for $k\geq 2$, an $R$-algebra is an algebra in the monoidal $\infty$-category $\RMod(R)$ of $R$-modules. 
For a spectrum $E$, we denote by $L_{E}$ the Bousfield localization functor at $E$.
For a spectrum $X$ and a pointed space $Y$, we write $X \otimes Y$ for the smash product $X \otimes \Sigma^{\infty}Y$.

\bigskip

\noindent \textsc{Acknowledgements.}
We are very grateful to Dustin Clausen for generously sharing his ideas and
numerous helpful discussions, and to Ben Antieau for his input and interest. We
also thank two anonymous referees for their comments. We heartily thank 
Ishan Levy for allowing us to include \Cref{prop:tstructure}. 
Finally, we thank Gijs Heuts, Niko Naumann, George Raptis, and Allen Yuan for valuable discussions and Ulrich Bunke and Lars Hesselholt for helpful comments on a previous version.

\section{Preliminaries}

\subsection{Preliminaries from chromatic homotopy theory}\label{recollections}

For an integer $n\geq 1$, we denote by $V_n$ a type $n$-complex, i.e.~a pointed finite CW-complex with 
$K(i)\otimes V_n=0$ for $i<n$ and $K(n)\otimes V_n \neq 0$. We denote by $v_n$  a $v_n$-self map of $V_n$, i.e.~a map $\Sigma^d V_n \to V_n$ for some positive integer $d$ inducing an isomorphism on $K(n)$-homology and nilpotent maps on $K(i)$-homology for $i\neq n$. Such maps exist by \cite{HS}. 

If $X$ is a pointed space or a spectrum, we define its $v_n$-periodic homotopy groups $v_n^{-1}\pi_*(X;V_n)$ by the formula 
\[v_n^{-1}\pi_*(X;V_n) = \Z[v_n^{\pm1}] \otimes_{\Z[v_n]} \pi_*\mathrm{Map}_*(V_n,X).\]

\begin{dfn}
We call a map of pointed spaces or spectra a \emph{$v_n$-periodic equivalence} (with $n\geq 1$) if it induces an isomorphism on $v_n$-periodic homotopy groups. A map of spectra or simple spaces is a \emph{$v_0$-periodic equivalence} if it becomes an equivalence after inverting $p$, and the $v_0$-periodic homotopy groups are by definition the homotopy groups with $p$ inverted. For a spectrum $E$, we say that another spectrum $X$ is $E$-\emph{acyclic} if $E\otimes X =0$ and say that a map is an $E$-\emph{equivalence} if its fibre is $E$-acyclic.
\end{dfn}

For a fixed pair $(V_n,v_n)$ we denote by $T(n)= \Sigma^\infty V_n[v_n^{-1}]$ the telescope of $v_n$. We adopt the convention that $T(0) = \bbS\adj$. We recall that the Bousfield class of a spectrum $E$ is the full subcategory of $\Sp$ consisting of the $E$-acyclic spectra.
For the convenience of the reader not familiar with chromatic homotopy theory, we note the following  well-known properties.

\begin{lemma}\label{lem:basicproperties} 
Let $X$ be a spectrum and $Y$ be a pointed space. 
\begin{enumerate}
	\item We have $v_n^{-1}\pi_*(X;V_n) \cong v_n^{-1}\pi_*(\Omega^{\infty}X;V_n)$.
	\item The maps $\tau_{\geq k}X \to X$ and $\tau_{\geq k}Y \to Y$ are $v_n$-periodic equivalences for all $k$ and all $n\geq 1$. 
	\item The spectra $K(m)$ are $T(n)$-acyclic for $n \neq m$.
	\item Any $T(n)$-acyclic spectrum is $K(n)$-acyclic.
	\item A spectrum which is $\bbS/p$-acyclic is also $T(n)$-acyclic for all $n\geq 1$.
	\item The map $X \to X^\swedge_p$ is a $T(n)$-equivalence for all $n \geq 1$.
	\item The Bousfield class of $T(n)$ does not depend on the choice of $(V_n,v_n)$.
	\item A spectrum is $T(n)$-acyclic if and only if its $v_n$-periodic homotopy groups vanish. 
\end{enumerate}
\end{lemma}
\begin{proof}
Part (i) follows immediately from the definitions and the equivalence $\Map_*(V_n,X) \simeq \Map_*(V_n,\Omega^\infty X)$. 
Assertion (ii) follows from the observation that the $v_n$-periodic homotopy groups of a bounded above spectrum or space vanish. This in turn follows from the fact that the degree $d$ of the self map $v_n$ is positive. 
Claim (iii) follows from the fact that $K(m)\otimes T(n) \simeq (K(m) \otimes V_n)[v_n^{-1}]$ which vanishes as $v_n$ is nilpotent on Morava $K$-homology if $n$ is different from $m$. 
To see (iv), assume that $X$ is $T(n)$-acyclic. Then we have $0 = K(n) \otimes T(n) \otimes X$. But $K(n)\otimes T(n) \neq 0$, so, since $K(n)$ is a field spectrum (any module is a direct sum of shifted copies of $K(n)$), we must have $K(n)\otimes X = 0$.

For part (v) observe that some power of $p$, say $p^k$, is zero on $V_{n}$ and hence on $T(n)$.
Given an $\bbS/p$-acyclic spectrum $X$, we have $X/p^k = 0$. Thus we see that
\[ 
0 = X/p^k \otimes T(n) \simeq X\otimes T(n)/p^k \simeq X \otimes (T(n)\oplus \Sigma T(n)),
\]
and the latter term has $X \otimes T(n)$ as a retract. Thus $X$ is $T(n)$-acyclic.
Statement (vi) follows from (v), since the fibre of $X \to X^\swedge_p$ is $\bbS/p$-acyclic.

For (vii), as in \cite[Lemma~2.1]{MR1328754}, we fix a pair $(V_n,v_n)$ and consider the full subcategory of finite $p$-local spectra consisting of those $Y$ which admit a $v_n$-self map $y$ and such that $T(n)\otimes Z=0$ implies that $Y[y^{-1}]\otimes Z = 0$ as well. 
This is a thick subcategory, as follows from \cite[Corollary 3.8]{HS}. Since it contains $V_n$, this thick subcategory is given by the thick subcategory of finite spectra of type at least $n$, see \cite[Theorem 7]{HS}. 
Hence if $T(n)\otimes Z = 0$ and $(V_n',v_n')$ is another choice, also $T(n)'\otimes Z = 0$. Running the same argument also with $V_n'$ instead of $V_n$ gives the claim.

To see (viii), consider a spectrum $X$ and observe that we may calculate its $v_n$-periodic homotopy groups using the mapping spectrum $\map(V_n,X)$ instead of the mapping space $\Map(V_n,X)$ due to the positivity of the degree of the self-map $v_n$. Thus, the $v_n$-periodic homotopy groups of $X$ are isomorphic to the homotopy groups of the spectrum $(DV_n\otimes X)[Dv_n^{-1}]$, where $DV_n$ denotes the dual of the finite spectrum $\Sigma^\infty V_n$ (which is again of type $n$). This spectrum is equivalent to $T(n)\otimes X$ where $T(n)$ is the telescope of $Dv_n$. The claim then follows from (vii).
\end{proof}

We remark that the converse of statement (iv) (for $n\geq 1$) is the content of the telescope conjecture. It is known 
\cite{Mahowald,
Miller} 
to be true in height $n=1$ and was recently disproved at all higher heights and all primes
\cite{BHLS}. 

We thank Dustin Clausen for help with the following lemma, which is a
consequence of the nilpotence theorem.  
\begin{lemma}\label{telescope conjecture on ring spectra}
Let $R$ be a ring spectrum, and $n \geq 1$ an integer. Then $R$ is $K(n)$-acyclic if and only if it is $T(n)$-acyclic.
\end{lemma}
\begin{proof}
The ``if''-part follows from \cref{lem:basicproperties}(iv). To see the ``only if'' statement, we argue first that one can assume that $T(n)$ is a ring spectrum. Indeed, by replacing $V_n$ by $W_n= V_n \otimes DV_n \simeq \End(\Sigma^\infty V_n)$, we can assume that the suspension spectrum of our type $n$-complex is an $\E_1$-ring spectrum. Moreover, the $v_n$-self-map of $V_n$ defines an element $w \in \pi_*(W_n)$, multiplication with which is a $v_n$-self map again. By \cite[Theorem 11]{HS} a power of $w$ lies in the center of $\pi_*(W_n)$. Thus the localization $W_{n}[w^{-1}]$ admits the structure of an $\E_1$-ring spectrum. As the Bousfield class of $T(n)$ does not depend on the choice of the type $n$ complex, we can thus indeed assume that $T(n)$ is a ring spectrum. 
We then observe that a ring spectrum like $T(n)\otimes R$ is zero if and only if its unit is nilpotent. By the nilpotence theorem \cite[Theorem 3]{HS}, this is the case if and only if $K(m) \otimes (T(n)\otimes R) = 0$ for all $0 \leq m \leq \infty$. If $m\neq n$ then $K(m) \otimes T(n) \otimes R=0$ as $K(m)\otimes T(n)=0$ by \cref{lem:basicproperties}(iii). Since $R$ is $K(n)$-acyclic, also $K(n)\otimes T(n)\otimes R = 0$. We thus conclude that $T(n) \otimes R$ vanishes.
\end{proof}

\begin{rem}
In the proof above, it was not used that $R$ is an algebra in the $\infty$-category of spectra. It suffices that $R$ is a unital magma in the homotopy category of spectra.
\end{rem}

\subsection{Some localization functors}

We recall that the functor $L_{n}^{f}$ on spectra is defined as Bousfield localization at the spectrum $H\Q\oplus T(1) \oplus \dots \oplus T(n)$. To formulate our main result, we will use the following variant.

\begin{dfn}
We denote by $L_{n}^{p,f}$ the Bousfield localization at $T(0) \oplus T(1) \oplus \dots \oplus T(n)$.
\end{dfn}

Recall that a Bousfield localization functor $L\colon \Sp \to \Sp$ is called smashing if it preserves colimits or, equivalently, if it is of the form $LX \simeq X \otimes L\bbS$. If every  $L$-acyclic spectrum is a colimit of compact, $L$-acyclic spectra, then $L$ is called finite and is in particular smashing, see \cite{Miller-finite} or \cite[Lecture 20, Example 12]{LurieLecturesChromatic}. For example, $L_{n}^{f}$ is smashing. The same proof also shows that $L_{n}^{p,f}$ is smashing. For convenience, we recall this proof below.
Write  $\cC_{>n}$ for the $\infty$-category of $p$-local, finite spectra which are of type greater than $n$, i.e.~which are $K(0) \oplus \dots \oplus K(n)$-acyclic.

\begin{lemma}
	\label{lem:Lnpf-is-finite}
The category of $L_{n}^{p,f}$-acyclic spectra coincides with $\Ind(\cC_{>n})$. In particular, $L_{n}^{p,f}$ is a smashing localization.
\end{lemma} 
\begin{proof}
The Bousfield class $\langle T(0) \oplus T(1) \oplus \dots \oplus T(n) \rangle$ has as complement the Bousfield class $\langle \Sigma^{\infty}V_{n+1} \rangle$ of a type $(n+1)$-spectrum: every spectrum is acyclic for $(T(0) \oplus T(1) \oplus \dots \oplus T(n)) \otimes V_{n+1}$ and $0$ is the only spectrum which is $T(0) \oplus T(1) \oplus \dots \oplus T(n) \oplus \Sigma^{\infty}V_{n+1}$-acyclic. Indeed, this follows easily from the inductive construction of a type $(k+1)$-complex as $V_{k}/v$, where $V_{k}$ is a type $k$-complex with $v_{k}$-self map $v$. It follows from \cite[Proposition 3.3]{MR1328754} that every $L_{n}^{p,f}$-acyclic spectrum is a colimit of finite $L_{n}^{p,f}$-acyclic spectra. The thick subcategory theorem implies that the category of finite $L_{n}^{p,f}$-acyclic spectra is precisely $\cC_{>n}$.
\end{proof}

\begin{lemma} 
	\label{lem:pullback-Lnpfs}
For integers $0 \leq m < n$ and a spectrum $X$ there is a  pullback diagram
\[
\begin{tikzcd}
 L_{n}^{p,f}X \ar[d]\ar[r] & L_{T(m+1) \oplus \dots \oplus T(n)}X \ar[d] \\ 
 L_{m}^{p,f}X \ar[r] & L_{m}^{p,f} L_{T(m+1) \oplus \dots \oplus T(n)}X
\end{tikzcd}
\]
natural in $X$.
\end{lemma}

\begin{proof}
This is a special case of  the following well known lemma.
\end{proof}

\begin{lemma}\label{localization-lemma}
Let $E$ and $F$ be spectra. Assume that $L_{E}$ preserves $F$-acyclic spectra. Then there is a pullback diagram 
\[
\begin{tikzcd}
 L_{E\oplus F}X \ar[d]\ar[r] & L_{F}X \ar[d] \\ 
 L_{E}X \ar[r] & L_{E}L_{F}X
\end{tikzcd}
\]
natural in $X$.
\end{lemma}
We note that the assumptions of the lemma are for instance satisfied if $L_E$ is smashing, or if $L_F$ annihilates $E$-local objects.
\begin{proof} 
Denote the pullback of the diagram $L_{E}X \to L_{E}L_{F}X \leftarrow L_{F}X$ by $P(X)$. There is a canonical map $X \to P(X)$; it is  easy to show that this map is an $(E\oplus F)$-local equivalence, and that $P(X)$ is $(E\oplus F)$-local.
\end{proof}

As a consequence of \cref{localization-lemma} we note that
\begin{enumerate}
\item for $p$-local spectra $X$, the canonical map $L_n^{p,f}X \to L_n^f X$ is an equivalence, and
\item for $T(1)$-acyclic spectra $X$, the canonical map $L_1^{p,f}X \to X\adj$ is an equivalence.
\end{enumerate}

We will use the following criterion to detect $T(i)$-local equivalences, which was indicated to us by 
Gijs Heuts.
\begin{prop}\label{T(i)-detects}
Let $f \colon X \to Y$ be a map of spectra, and let $i \geq 1$ be an integer. If $\Sigma^\infty\Omega^\infty f$ is a $T(i)$-local equivalence, then so is $f$. In other words, the functor $\Sigma^\infty \Omega^\infty \colon \Sp \to \Sp$ detects $T(i)$-local equivalences.
\end{prop}
\begin{proof}
We note that the canonical composite 
\[ \Omega^\infty \lto \Omega^\infty\Sigma^\infty \Omega^\infty \lto \Omega^\infty \]
is an equivalence. It is an insight of Bousfield and Kuhn that the $T(i)$-localization functor $L_{T(i)}$ factors through $\Omega^\infty$ via the Bousfield--Kuhn functor $\Phi_i$ from pointed spaces to spectra. There is thus an equivalence $\Phi_i \circ \Omega^\infty \simeq L_{T(i)}$; see \cite[Theorem 1.1]{Kuhn1}. Applying $\Phi_i$ to the above composite shows that the composite 
\[ L_{T(i)} \lto L_{T(i)} \Sigma^\infty \Omega^\infty \lto L_{T(i)} \]
is also an equivalence. This implies that $L_{T(i)}(f)$ is a retract of $L_{T(i)}(\Sigma^\infty\Omega^\infty f)$ which proves the lemma.
\end{proof}

\begin{rem}
Restricted to connective spectra, the functor $\Sigma^\infty\Omega^\infty$ also detects $T(0)$-local equivalences. 
\end{rem}

It is, however, not true that the functor $\Sigma^\infty\Omega^\infty$ \emph{preserves} $T(i)$-local equivalences. For example, $H\Z$ is $T(i)$-acyclic, whereas $\Sigma^\infty \Z$ is not: It contains the sphere spectrum as a summand. Nevertheless, $\Sigma^\infty\Omega^\infty$  preserves  suitably connective $L_{n}^{p,f}$-equivalences, as the following result shows.
We say that a space is $m$-connective if it has trivial homotopy groups in
degrees less than $m$, i.e., is $(m-1)$-connected. 
We call a map $m$-connective if the fibre over every basepoint is $m$-connective, i.e.\ if it induces an isomorphism on $\pi_k$ for $k<m$ and a surjection on $\pi_m$. 

\begin{prop}\label{lem:vnTnvanishing}
Let $n \geq 1$ be an integer. Then there exists $m \geq 2$ such that the
following hold: 
\begin{enumerate}
\item Let $F$ be an $m$-connective pointed space whose $v_i$-periodic homotopy groups vanish for $0\leq i\leq n$. Then $F$ is $T(i)$-acyclic for $0 \leq i \leq n$.
\item Let $f \colon X \to Y$ be an $m$-connective map between spaces. If $f$ is a $v_i$-periodic equivalence for $0 \leq i \leq n$ for every choice of basepoints, then $\Sigma^\infty f \colon \Sigma^\infty X \to \Sigma^\infty Y$ is an $L_n^{p,f}$-equivalence for every choice of basepoints.
\item The functor $\Sigma^\infty \Omega^\infty$ preserves $m$-connective $L_n^{p,f}$-equivalences.
\end{enumerate}
\end{prop}
\begin{proof}
Part (i) follows from a result of Bousfield (\cite[Corollary 4.8]{Bousfield2}, \cite[Theorem~3.1]{BHM}) together with \cite[Lemma~3.3]{BHM}, which
gives an integer $m$ such that any $m$-connective space with trivial $v_i$-periodic
homotopy groups for $0 \leq i \leq n$ has trivial $T(i)$-homology for $0 \leq
i \leq n$. Note that in Bousfield's convention $T(0)$ is $H\mathbb{Q}$, but it is well-known that a simply connected space such that $(\pi_*X)[\frac1p]$ vanishes is also acyclic for $H\Z[\frac1p]$ and hence for $T(0) = \mathbb{S}[\frac1p]$.

To prove (ii), consider the Serre spectral sequence in $T(i)$-homology associated to the map $X \to Y$. By assumption, the fibre over every basepoint is an $m$-connective space whose $v_i$-periodic homotopy groups vanish for $i \leq n$; thus, it is $T(i)$-acyclic for $0 \leq i \leq n$ by (i). It follows that $X \to Y$ induces an isomorphism in $T(i)$-homology for $0\leq i \leq n$. Hence, $\Sigma_+^\infty X \to \Sigma_+^\infty Y$ and thus  $\Sigma^\infty X \to \Sigma^\infty Y$ (for any choice of basepoints) are $L_n^{p,f}$-equivalences.

Finally, to see (iii) one applies $\Omega^\infty$ to an $m$-connective $L_n^{p,f}$-equivalence. The resulting map of spaces satisfies the assumptions of (ii), so the proposition follows.
\end{proof}

The following remark will not be used in the sequel. 
\begin{rem} 
In fact, we can take $m = n+1$ in \cref{lem:vnTnvanishing}. 
Indeed, using what we have proven already, it suffices to show that 
 any $(n+1)$-connective space $F$ is 
 $T(i)$-acyclic for $0 \leq i \leq n$ if its homotopy groups are $p$-primary torsion and 
vanish in sufficiently high degrees. By an induction over the Postnikov tower,
it hence suffices to show that $T(i) \otimes K(\pi,r) = 0$ if $r > i$ and $\pi$
is a finite group, as $T(i)$-homology commutes with filtered colimits and every
torsion group is the filtered colimit of its finite subgroups. It is shown in
\cite[Theorem E]{CSYAmbi} that for a $p$-local ring spectrum $R$ the following
two conditions are equivalent: (1) $R \otimes K(\pi,r) = 0$ for all $r > i$ and
(2) $R\otimes K(r) = 0$ for all $r > i$. Statement (2) applies to $R=T(i)$ by
\cref{lem:basicproperties}(iii), so part (i) and hence also (ii) and
(iii) of the proposition follow with $m=n+1$.
\end{rem}

\subsection{Localizing invariants and \textit{K}-theory}
	\label{sec:loc-and-K}
In this short subsection we recall some notions and facts about algebraic $K$-theory which we will use throughout this paper.

A localizing invariant is a functor $\Cat_\infty^\perf \to \Sp$ which sends exact sequences to fibre sequences. Here $\Cat_\infty^\perf$ refers to the $\infty$-category of small, idempotent complete, and stable $\infty$-categories and exact sequences are those sequences which are both fibre and cofibre sequences in $\Cat_\infty^\perf$, see \cite[\S 5]{BGT} for details.\footnote{In \cite{BGT} localizing invariants are further required to preserve filtered colimits.}
Examples of localizing invariants are non-connective $K$-theory \cite[\S9]{BGT}, topological Hochschild homology, topological cyclic homology, etc. For a localizing invariant $E$ and a ring spectrum $A$, we will write $E(A)$ for $E(\Perf(A))$, where $\Perf(A)$ denotes the $\infty$-category of perfect $A$-modules, which coincides with the compact objects of $\RMod(A)$. 

For a connective ring spectrum $A$, the space $\Omega^\infty \tau_{\geq 1} K(A)$ can be described as a plus-con\-struction \cite[Lemma 9.39]{BGT}: We denote by $\GL(A)$ the $\E_1$-space $\GL(A) = \colim \GL_n(A)$, where $\GL_n(A)$ denotes the invertible components in the $\E_1$-space $\Omega^\infty \End(A^n)$. 
In particular, $\pi_{0}(\GL(A)) \cong \GL(\pi_{0}(A))$, where the right-hand side denotes the group of invertible matrices over the discrete ring $\pi_{0}(A)$.
There is a canonical map $\BGL(A) \to \Omega^\infty \tau_{\geq 1} K(A)$ which exhibits the target as the plus construction $\BGL(A)^+$. In particular, this map is a homology equivalence, and hence the map of spectra $\Sigma^\infty \BGL(A) \to \Sigma^\infty\Omega^\infty \tau_{\geq1} K(A)$ is an equivalence.

For an arbitrary $\mathcal{C} \in \Cat_\infty^{\perf}$, 
we will also need the explicit description of the $K$-theory space
$\Omega^\infty  K(\mathcal{C})$, which arises via the Waldhausen
$S_\bullet$-construction, cf.~\cite[Sec.~7.1]{BGT}. 
The $S_\bullet$-construction produces a simplicial object $S_\bullet \mathcal{C}
\in \fun(\Delta^{\op}, \Cat_\infty^{\perf})$ such that there is a natural equivalence
of spaces
\begin{equation}
\Omega^\infty K(\mathcal{C})  \simeq \Omega |
S_\bullet(\mathcal{C})^{\simeq}|, \label{geomrealexpr} 
\end{equation}
where $(-)^{\simeq}$ denotes the underlying space of an $\infty$-category;
moreover, we have for each $n \geq 0$ a natural equivalence $S_n(\mathcal{C})
\simeq \mathrm{Fun}( \Delta^{n-1},
\mathcal{C})$.  
Note that both sides have the canonical structure of
$\mathbb{E}_\infty$-spaces since $\Cat_\infty^{\perf}$
is semiadditive, i.e.\ it has finite biproducts as in \cite[Definition 2.1]{GGN15}; In fact, the equivalence \ref{geomrealexpr} is one of $\E_\infty$-spaces, therefore, we can deloop both sides to obtain 
\begin{equation} \label{deloopedK}
\Omega^\infty( \tau_{\geq 0} K(\mathcal{C})[1]) \simeq
|S_\bullet(\mathcal{C})^{\simeq} |.
\end{equation}

\section{Proof of the Purity Theorem}

In this section, we will prove \Cref{thm-A}, which we restate here as \Cref{purity1}, in several steps, each of which will give a
special case of the result. We then combine this with Theorem~\ref{thm-cmnn} to obtain the Purity Theorem.
Our first step, which we treat in the following subsection, will
involve only highly connective maps of connective ring spectra. 

\subsection{The highly connective case}
	\label{sec:proof-of-main-thm}
	
\begin{prop}\label{prop:E(n)-local-K-theory}
Let $n \geq 1$. There exists $N \geq 1$ such that the following holds:  let $A \to B$ be an $N$-connective $L_n^{p,f}$-equivalence between connective ring spectra. Then the induced map $K(A) \to K(B)$ is again an $L_n^{p,f}$-equivalence.
\end{prop}

\begin{proof}
We take $N = m-1$, where $m$ is as in 
\Cref{lem:vnTnvanishing}. 
By Waldhausen's result (see \cite[Propositions  1.1, 2.2]{Waldhausen} or \cite[Lemma 2.4]{LT}) the map $K(A) \to K(B)$ is a  $T(0)$-equivalence. It hence remains to prove that the map $K(A) \to K(B)$ is a $T(i)$-local equivalence for $1 \leq i \leq n$. By \cref{lem:basicproperties}(ii), it suffices to show that $\tau_{\geq 1}K(A) \to \tau_{\geq 1} K(B)$ is a $T(i)$-local equivalence for $1 \leq i \leq n$. 

We consider the following commutative diagram, where we use the plus-construction description of algebraic $K$-theory for connective ring spectra as recalled in \cref{sec:loc-and-K}.
\[\begin{tikzcd}
	\Sigma^\infty \BGL(A) \ar[r] \ar[d,"\simeq"'] & \Sigma^\infty \BGL(B) \ar[d,"\simeq"] \\
	\Sigma^\infty\Omega^\infty \tau_{\geq 1}K(A) \ar[r] & \Sigma^\infty\Omega^\infty \tau_{\geq 1} K(B)
\end{tikzcd}\]
By \cref{T(i)-detects}, it suffices to show that the lower horizontal map is a $T(i)$-local equivalence for $1 \leq i \leq n$. Since the vertical maps in the above diagram are equivalences, this is the case if the top horizontal map is a $T(i)$-local equivalence. This and thus the proposition will follow from \cref{lem:vnTnvanishing}(ii) once we have shown the following: The map $\BGL(A) \to \BGL(B)$ is $m$-connective and induces an isomorphism on $v_{i}$-periodic homotopy groups for $i\leq n$.

To show this claim, we observe that the classifying space construction $\B$
increases the connectivity of a map by $1$ and preserves $v_{i}$-periodic
equivalences. Thus it suffices to see that $\GL(A) \to \GL(B)$ is an $(m-1)$-connective $v_{i}$-periodic equivalence for $1\leq i \leq n$.
Observe that by definition the map $A\to B$ induces an isomorphism between $\pi_{0}(\GL(A)) = \GL(\pi_{0}(A))$ and $\GL(\pi_{0}(B)) = \pi_{0}(\GL(B))$. Furthermore,  
 $\tau_{\geq 1}\GL(A) \simeq \tau_{\geq1}\M(A)$, where $\M(A)$ is the space of matrices $\colim_{r} \Omega^{\infty}\End(A^{r}) \simeq \colim_{r} \Omega^{\infty}A^{r\times r}$, and similarly for $B$. 
As $A \to B$ is $(m-1)$-connective and as $m\geq 2$, we thus see that $\GL(A)
\to \GL(B)$ is $(m-1)$-connective. Further, as $A \to B$ is an $L_{n}^{p,f}$-equivalence, it induces isomorphisms on $v_{i}$-periodic homotopy groups for $i\leq n$. By Lemma~\ref{lem:basicproperties}(i) also $\M(A) \to \M(B)$ is a $v_{i}$-periodic equivalence, and, finally, by Lemma~\ref{lem:basicproperties}(ii) also $\GL(A) \to \GL(B)$ is a $v_{i}$-periodic equivalence for $i\leq n$, as desired.
\end{proof}

\subsection{A truncating property}
Our next goal is to prove a version of \Cref{prop:E(n)-local-K-theory} with weaker connectivity
hypotheses; in fact, we will only need a special case
(\Cref{TnKistruncatingT0Tnacyclic}) below, formulated using the language of
truncating invariants. 
To do this, we will need some further preliminaries.

\begin{lemma}	
	\label{lem:connectivity-of-multiplication}
Let
\[
\begin{tikzcd}
 A \ar[d]\ar[r] & B \ar[d] \\ 
 A' \ar[r] & B' 
\end{tikzcd}
\]
be a pullback square of ring spectra in which $A$ is connective. If $A \to A'$ is $n$-connective and $A \to B$ is $m$-connective, then the induced map $A' \wtimes{A}{B'} B \to B'$ is $(m+n+2)$-connective.
\end{lemma}
Here $A' \wtimes{A}{B'} B$  denotes the ring spectrum associated to the displayed pullback square by \cite[Main Theorem]{LT}.
\begin{proof}
Denote  by $I$ the common fibre of the vertical maps, by $J$ that of the horizontal maps. From \cite[Remark 1.16]{LT} we have an equivalence
\[
\fib( A' \wtimes{A}{B'} B \to B' ) \simeq \Sigma \fib( I \otimes_{A} B \xrightarrow{\mu} I )
\]
where the map $\mu$ is induced by the right $B$-module structure on $I$. Now $\mu$ has a section $\sigma\colon I \to I \otimes_{A} B$ induced by the map $A \to B$. The fibre of $\sigma$ identifies with $I \otimes_{A} J$, which is $(n+m)$-connective as $A$ is connective. In other words, $\sigma$ is an isomorphism in degrees $\leq m+n-1$ and surjective in degree $m+n$. Since $\mu\circ\sigma\simeq \id_{I}$ and so $\mu$ is surjective in every degree, it then follows that $\mu$ is an isomorphism in degrees $\leq m+n$ and surjective in degree $m+n+1$, i.e.~$\mu$ is $(m+n+1)$-connective. By the above equivalence, $A' \wtimes{A}{B'} B \to B'$ is then $(m+n+2)$-connective, as desired.
\end{proof}

Let $M$ be a spectrum. 
We say that a localizing invariant $E$ is \emph{truncating on $M$-acyclic ring
spectra} if for every $M$-acyclic connective ring spectrum $R$, we have
$E(R) \xrightarrow{\sim} E(\pi_0 R)$. Note that if a ring spectrum $R$ is $M$-acyclic, then also $\tau_{\leq k}R$ is $M$-acyclic for all $k$ as follows by consideration of the ring map $L_MR \to L_M{\tau_{\leq k} R}$. The following lemma also appears in similar form in \cite[Lemma~3.11]{MTR}.

\begin{lemma} 
\label{ktruncatingimpliestruncating}
Let $E$ be a localizing invariant. 
Suppose that there exists $k \geq 0$ such that the map 
$E(R) \xrightarrow{\sim} E(\tau_{\leq k} R)$ is an equivalence for any $M$-acyclic connective ring spectrum $R$. Then $E$ is
truncating on $M$-acyclic ring spectra. 
\end{lemma}
\begin{proof} 
It suffices to show that if $E$ and $k>0$ are as in the statement of the lemma,
then we have in fact $E( \tau_{\leq k} R) \xrightarrow{\sim} E(\tau_{\leq k-1}
R)$ for all $M$-acyclic connective ring spectra $R$; the result then follows by induction on $k$. 

To this end, recall that 
$\tau_{\leq k} R \to \tau_{\leq k-1} R$ is a square-zero extension, so 
there is a pullback square of ring spectra (cf.~\cite[7.4.1.29]{HA}),
\[ 
\begin{tikzcd}
\tau_{\leq k} R \ar[d]  \ar[r] &  H \pi_0 R \ar[d]  \\
\tau_{\leq k-1} R \ar[r] &  H \pi_0 R \oplus (H \pi_k R )[k+1].
\end{tikzcd}
\]
All ring spectra in this square are connective and $M$-acyclic. 
It follows 
from 
\cite[Main Theorem]{LT}
that we have a pullback square of spectra
\[ 
\begin{tikzcd}
E( \tau_{\leq k} R) \ar[d]  \ar[r] &  E( H \pi_0 R) \ar[d]  \\
E( \tau_{\leq k-1} R) \ar[r] &  E\left( \tau_{\leq k-1} R \wtimes{\tau_{\leq k} R }{H
\pi_0 R \oplus (H \pi_k R)[k+1]} H \pi_0 R \right).
\end{tikzcd}
\]
It thus suffices to show that the right vertical map is an equivalence. 
But this follows because, by \Cref{lem:connectivity-of-multiplication}, the map 
of connective, $M$-acyclic ring spectra
\[ H \pi_0 R \to \tau_{\leq k-1} R \wtimes{\tau_{\leq k} R }{H
\pi_0 R \oplus (H \pi_k R)[k+1]} H \pi_0 R  \]
induces an equivalence on $\tau_{\leq k}$ and hence on $E(-)$. 
\end{proof}

\begin{prop}\label{TnKistruncatingT0Tnacyclic}
For $n \geq 1$, $L_{T(n)} K(-)$ is truncating on $L_{n}^{p,f}$-acyclic ring spectra. 
\end{prop} 
\begin{proof} 
This follows from \Cref{prop:E(n)-local-K-theory} and 
\Cref{ktruncatingimpliestruncating}. 
\end{proof}

\begin{cor}[Cf.~also \cite{BCM}] 
For any $n \geq 1$, we have $L_{T(i)} K(\mathbb{Z}/p^n) = 0$ for $i \geq 1$. 
\label{K1KZpn}
\end{cor} 
\begin{proof} 
This follows 
from \Cref{TnKistruncatingT0Tnacyclic} since truncating invariants are
nilinvariant, \cite[Corollary~3.5]{LT}, and Quillen's computation that $K(\F_{p})^{\swedge}_{p}=H\Z_{p}$.
\end{proof} 

\subsection{The general case}
Now we extend the results to nonconnective ring spectra, and then
complete the proofs of \Cref{thm-A} and the Purity Theorem. Our strategy of proof is to reduce the nonconnective case to the
connective case using the $S_\bullet$-construction. In this we will work with a not-necessarily stable, but additive $\infty$-category $\cA$, about which we make two remarks: 
\begin{itemize}
	\item We can view $\cA$ as a symmetric monoidal $\infty$-category under $\oplus$ and denote by $K^{\mathrm{add}}(\cA)$ its group-completion $K$-theory, cf.~\cite{GGN15} for a modern account. If $R$ is a connective ring spectrum and $\cA$ is the $\infty$-category $\mathrm{Proj}^\omega(R)$ of finitely generated projective $R$-modules, there is an equivalence $K^{\mathrm{add}}(\cA) \simeq \tau_{\geq 0}K(R)$.
	\item Given two objects $X$ and $Y$ in an additive $\infty$-category $\cA$, their mapping space has the canonical structure of a grouplike $\E_{\infty}$-space, giving rise to a connective spectrum $\hom_{\cA}(X,Y)$. If $\cA$ is stable, this is the connective cover of the homomorphism spectrum $\mathrm{Hom}_{\cA}(X,Y)$. We remark that $\mathrm{Hom}_{\cA}(X,X)$ is $L_n^{p,f}$-acyclic if and only if $\mathrm{hom}_{\cA}(X,X)$ is $L_n^{p,f}$-acyclic: for $T(i)$-acyclicity with $i>0$, this is \cref{T(i)-detects} and each of $\mathrm{hom}_{\cA}(X,X)$ and $\mathrm{Hom}_{\cA}(X,X)$ is $T(0)$-acyclic if and only if $[\id_X]$ in the common $\pi_0$ is $p$-power torsion.
\end{itemize}
In the following, we will assume that all $\infty$-categories 
of which we consider the $K$-theory are idempotent-complete.
\begin{prop}  
Let $\mathcal{C}$ be an additive $\infty$-category. Suppose 
for each object $X \in \mathcal{C}$, the ring spectrum $\hom_{\mathcal{C}}(X, X)$ is
annihilated by $L_n^{p,f}$. Then 
\begin{enumerate}
	\item $L_{T(i)} K^{\mathrm{add}}(\mathcal{C}) = 0$ for $ 1 \leq i \leq n$, and
	\item $L_{T(i)} K(\mathcal{C}) = 0$ for $ 1 \leq i \leq n$ if $\mathcal{C}$ is stable. 
\end{enumerate}
\label{vanishTilocal}
\end{prop} 
\begin{proof}

For the first part, we observe that $\mathcal{C}$ can be written as a filtered colimit of its full subcategories generated by finite direct sums and retracts by finitely many objects. Passing to the direct sum of the generators, and using that $K$-theory commutes with filtered colimits, we may assume that $\mathcal{C}$ is generated under finite direct sums
and retracts by a single object $X$. Hence, by the additive version of the Schwede--Shipley theorem, $\mathcal{C} \simeq
\mathrm{Proj}^\omega(
\hom_{\mathcal{C}}(X, X))$ is the $\infty$-category of finitely generated
projective modules over $\hom_{\mathcal{C}}(X, X)$, which is $L_n^{p,f}$-acyclic by assumption. 
Therefore, $K^{\mathrm{add}}(\mathcal{C}) \simeq \tau_{\geq 0}K(\hom_{\mathcal{C}}(X,X))$ is $T(i)$-acyclic for $1 \leq i \leq n$ by
\Cref{TnKistruncatingT0Tnacyclic}
(together with the fact that the $T(1)$-local $K$-theory of a $p$-power torsion discrete ring
vanishes by \Cref{K1KZpn}).

For the second part, we assume that $\mathcal{C}$ is stable.
The Waldhausen $S_\bullet$-construction gives a 
simplicial stable $\infty$-category $S_\bullet \mathcal{C}$ and a natural
equivalence of spaces (as in \eqref{deloopedK}):
\[ \Omega^\infty\left(  (\tau_{\geq 0}K(\mathcal{C}))[1]\right) = |  S_\bullet(\mathcal{C})^{\simeq} |.  \]
Note that both sides have the structure of $\mathbb{E}_\infty$-monoids, using
the direct sum on $\mathcal{C}$ (which also gives the canonical
$\mathbb{E}_\infty$-monoid structure on the left arising from $\Omega^\infty$),
and the map is an equivalence of $\mathbb{E}_\infty$-spaces. 
Therefore, we may group-complete the terms inside the geometric realization on
the right-hand-side to obtain
an equivalence of connective spectra
\begin{equation} (\tau_{\geq 0} K(\mathcal{C}))[1] \simeq |
K^{\mathrm{add}}(S_\bullet(\mathcal{C}))|,   \label{equivconnectivesp} \end{equation}
where on the right we consider the additive (group-complete) $K$-theory
as above. 
The above $L_n^{p,f}$-local vanishing condition on the stable $\infty$-category
$\mathcal{C}$ 
is stable under passage to $\fun(\Delta^j, \mathcal{C})$ for any
$j \geq 0$. Therefore, by the first paragraph of the proof, we find that the 
right-hand-side of 
\eqref{equivconnectivesp}
is $T(i)$-acyclic for $1 \leq i \leq n$, hence the result. 
\end{proof} 

For an alternative argument that (i) implies (ii) in the above Proposition, see Proposition~\ref{prop:tstructure}.

\begin{lemma}\label{lem:LnpfLocalizationSequence}
	For any ring spectrum $A$, there is an exact sequence
	\[ 
	\cC_{>n} \otimes \Perf(A) \lto \Perf(A) \lto \Perf(L_n^{p,f}A),
	\]
	and the endomorphism spectrum of every object in $\cC_{>n} \otimes \Perf(A)$ is $L_n^{p,f}$-acyclic. 
\end{lemma}
\begin{proof}
Lemma~\ref{lem:Lnpf-is-finite} and the Thomason--Neeman localization theorem \cite[Theorem 2.1]{MR1191736} imply that the sequence of small stable $\infty$-categories
\[
\cC_{>n} \lto \Perf(\bbS) \lto \Perf(L_{n}^{p,f}\bbS)
\]
is exact. Tensoring the above exact sequence with $\Perf(A)$, using the fact that $L_{n}^{p,f}$ is smashing, we obtain the exact sequence
\[ 
\cC_{>n} \otimes \Perf(A) \lto \Perf(A) \lto \Perf(L_n^{p,f}A).
\]
The $\infty$-category $\cC_{>n} \otimes \Perf(A)$ is generated by $A$-modules of the form $A \tensor F$ with $F$ being a finite $p$-local $L_n^{p,f}$-acyclic spectrum. Their endomorphism spectra $DF \tensor F \tensor A$ are $L_n^{p,f}$-acyclic as well and so are thus the endomorphism spectra of all objects of $\cC_{>n} \otimes \Perf(A)$. 
\end{proof}

The following is a reformulation of \Cref{thm-A}.
\begin{thm} 
\label{purity1}
Let $A$ be a ring spectrum. Then the map $A \to L_n^{p,f} A$ induces an
equivalence on $L_{T(i)} K(-)$ for $1 \leq i \leq n$. 
If $n \geq 2$, the map 
$A \to L_{T(1) \oplus \dots \oplus T(n)} A$ induces an equivalence on $L_{T(n)}
K(-)$. 
\end{thm} 
\begin{proof}
As $K$-theory is localizing, the homotopy fibre of $K(A) \to K( L_n^{p,f} A)$ coincides with $K(\cC_{>n} \otimes \Perf(A))$ by the preceding lemma. This is $T(i)$-acyclic for $1\leq i\leq n$ by \Cref{vanishTilocal}.
For the last assertion, we consider the pullback diagram
\[\begin{tikzcd}
	L_n^{p,f}A \ar[r] \ar[d] & L_{T(1) \oplus \dots \oplus T(n)} A \ar[d] \\
	A[\tfrac{1}{p}] \ar[r] & (L_{T(1) \oplus \dots \oplus T(n)} A)[\tfrac{1}{p}]
\end{tikzcd}\]
from \cref{lem:pullback-Lnpfs} (note that $L_0^{p,f}A = A[\tfrac{1}{p}]$). By \cite{Tamme} or \cite{LT} and the fact that $L_0^{p,f}$ is smashing, we deduce that the diagram
\[\begin{tikzcd}
	K(L_n^{p,f}A) \ar[r] \ar[d] & K(L_{T(1)\oplus \dots \oplus T(n)}A) \ar[d] \\
	K(A[\tfrac{1}{p}]) \ar[r] & K((L_{T(1)\oplus \dots \oplus T(n)}A)[\tfrac{1}{p}])
\end{tikzcd}\]
is a pullback.\footnote{This also follows more classically from \cref{lem:LnpfLocalizationSequence}.} By the first part, it suffices to prove that the top horizontal map is an equivalence after $T(i)$-localization for $i\geq 2$. Now each term in the bottom row is a module over $K(\bbS[\tfrac{1}{p}])$, which is $p$-adically equivalent to $K(\mathbb{Z}[\tfrac{1}{p}])$ and hence vanishes after $T(i)$-localization for $i\geq 2$ by Mitchell's result \cite{Mitchell}. 
\end{proof}

We now prove the Purity Theorem. 
The result is a direct consequence of \Cref{purity1} (which proves
``one half'' of the result) and Clausen--Mathew--Naumann--Noel's \Cref{thm-cmnn} 
(which proves the ``other half''). 
We note that the results of \cite{CMNN2}
also rely 
on \Cref{purity1} (but not on the Purity Theorem), so there is no
circularity. 

\begin{proof}[Proof of the Purity Theorem] 
We have to show that the map $A \to L_{T(n-1)\oplus T(n)}A$ induces an equivalence on $L_{T(n)}K(-)$.
As in the proof of \cref{purity1}, using that $L_{n-2}^{p,f}$ is a smashing localization (or \cref{lem:LnpfLocalizationSequence}), the diagram
\[\begin{tikzcd}
	K(L_n^{p,f}A) \ar[r] \ar[d] & K(L_{T(n-1)\oplus T(n)}A) \ar[d] \\
	K(L_{n-2}^{p,f}A) \ar[r] & K(L_{n-2}^{p,f}(L_{T(n-1)\oplus T(n)}A)) 
\end{tikzcd}\]
is a pullback. By \cref{purity1}, it suffices to prove that the top horizontal map is an equivalence after $T(n)$-localization.
Now each term in the bottom row is a module over $K(L_{n-2}^{p,f}\bbS)$, which
vanishes $T(n)$-locally by \Cref{thm-cmnn}, so we deduce the claim.
\end{proof}

\begin{rem}\label{CMNN-argument}
In \cite{CMNN2}, the vanishing of $L_{T(n+2)}K(L_n^{p,f}\bbS)$ is deduced as a
special case of a more general result. For the convenience of the reader, we
summarize their argument for the exact vanishing that we use here. We wish
to show the claim by induction over $n$. The case $n=0$ follows from Mitchell's
theorem, as in the proof of the second part of \cref{purity1}. By the
strengthening of Hahn's result \cite{Hahn} obtained in \cite[Lemma 4.5]{CMNN2}, it suffices to show that $K(L_n^{p,f}\bbS)^{tC_p}$ is $T(n+1)$-acyclic. Now, given any commutative algebra in genuine $C_p$-spectra whose underlying spectrum with $C_p$-action is $K(L_n^{p,f}\bbS)$ with trivial $C_p$-action, there is a ring map from its geometric fixed points to $K(L_n^{p,f}\bbS)^{tC_p}$. An example is the $K$-theory of the Borel complete categorical Mackey functor, where the genuine fixed points are $K(\Fun(BC_p,\Perf(L_n^{p,f}\bbS)))$. The transfer for this genuine $C_p$-spectrum is the composite
\begin{equation}\label{transfer}
K(L_n^{p,f}\bbS)_{hC_p} \lto K(L_n^{p,f}\bbS[C_p]) \lto K(\Fun(BC_p,\Perf(L_n^{p,f}\bbS))),
\end{equation}
and by definition the geometric fixed points are the cofibre of this composite.
It hence suffices to show that each of the above two maps is a $T(n+1)$-local
equivalence. For the second, one uses that the Verdier quotient
$\Fun(BC_p,\Perf(L_n^{p,f}\bbS))/\Perf(L_n^{p,f}\bbS[C_p])$ is linear over
$(L_n^{p,f}\bbS)^{tC_p}$. Indeed, calling this quotient $\mathcal{Q}$ and writing $R = L_n^{p,f}\mathbb{S}$, Theorem I.3.3ii and an analogue of Lemma I.3.8iii from \cite{NikolausScholze} imply that $\End_{\mathcal{Q}}(R) \simeq R^{tC_p}$. By Theorem I.3.6 in op.\ cit., $\mathcal{Q}$ has a canonical symmetric monoidal structure and we obtain a symmetric monoidal functor $\Perf(R^{tC_p}) \to \mathcal{Q}$. The spectrum $(L_n^{p,f}\bbS)^{tC_p}$ is itself an algebra over $L_{n-1}^{p,f}\bbS$ by
Kuhn's blueshift result \cite{Kuhn}. Thus, by induction, the second map in \eqref{transfer} is a $T(n+1)$-local equivalence.

For the first map one uses \cref{compKTC} (whose proof relies only on
\cref{purity1}) to obtain a diagram which is cartesian after
$T(i)$-localization, $i \geq 2$,
\[\begin{tikzcd}
	K(L_n^{p,f}\bbS)_{hC_p} \ar[r] \ar[d] & K(L_n^{p,f}\bbS[C_p]) \ar[d] \\
	\TC(\tau_{\geq 0}(L_n^{p,f}\bbS))_{hC_p} \ar[r] & \TC(\tau_{\geq 0}(L_n^{p,f}\bbS)[C_p]).
\end{tikzcd}\]
Now, by \cite[Theorem 1.4.1]{HN}, the cofibre of the lower horizontal map
belongs to the localizing subcategory of spectra generated by  $\tau_{\geq
0}(L_n^{p,f}\bbS)$, and hence vanishes $T(n+1)$-locally as well.
\end{rem}

\begin{quest} 
For a ring spectrum $A$ and for $n \geq 2$, does the map $A \to L_{K(n-1)
\oplus K(n)} A$ induce an equivalence on $K(n)$-local $K$-theory? 
\end{quest} 

The above question reduces to proving an analog of \Cref{purity1} for $L_n$-localization: that is,
for $n \geq 2$, it would suffice to show that $A \to L_n A$ induces an equivalence on
$L_{K(n)} K(-)$.

\begin{rem} 
\label{purity:optimal} Suppose $n \geq 1$. The 
functor $A \mapsto L_{T(n)} K(A)$ does not factor through either of the further
localizations $A \mapsto L_{T(n)} A$ or $A \mapsto L_{T(n-1)} A$; in this
sense, the purity theorem is optimal. 

In fact, 
the results of \cite{HW21, Yuan} give examples of $T(n)$-acyclic ring spectra
for which $L_{T(n)} K(A) \neq 0$. 
This shows that 
$A \mapsto L_{T(n)} K(A)$ does not factor through $A \mapsto L_{T(n)} A$. 
Moreover, we claim that 
$A \mapsto L_{T(n)} K(A)$  does not factor through 
$A \mapsto L_{T(n-1)} A$. 
Suppose the contrary: that is, suppose that any $T(n-1)$-equivalence of
$\mathbb{E}_1$-rings induced an equivalence on $L_{T(n)} K(-)$. 
We will obtain a contradiction as follows. 
Fix any connective spectrum $M$ which is $T(n-1)$-acyclic but not
$T(n)$-acyclic. 
Under our assumption, for any $r \geq 0$,
the map $\mathbb{S} \oplus \Sigma^r M \to \mathbb{S}$ of $\mathbb{E}_1$-rings
(where the former denotes the square-zero extension) induces an equivalence on
$L_{T(n)} K(-)$. It follows that 
the colimit 
$\colim_r \Omega^r\mathrm{fib} ( K( \mathbb{S} \oplus \Sigma^r M) \to
K(\mathbb{S}))$ is $T(n)$-acyclic. But by the equivalence of stable $K$-theory
with topological Hochschild homology \cite[Theorem~5.3.5.1]{DGM}, it follows that this colimit is $\Omega M$,
yielding a contradiction. 
\end{rem} 

We finish this section with an alternative proof of the implication (i) $\Rightarrow$ (ii) in \cref{vanishTilocal} which was explained to us by Ishan Levy. We thank him for suggesting to include this argument here. It is closely related to work of Neeman \cite{Neeman21}.

To introduce necessary notation, let us denote for an additive $\infty$-category $\cC$ by $\Stab(\cC)$ its stable envelope, i.e.\ the compact objects of $\Fun^\add(\cC^\op,\Sp)$. The image of the fully faithful Yoneda embedding $y\colon \cC \to \Stab(\cC)$, using the connective delooping of the mapping spaces given by additivity of $\cC$, generates $\Stab(\cC)$ as an idempotent complete stable $\infty$-category. Moreover, $\tau_{\geq0}K(\Stab(\cC))$ identifies with the group completion of $\cC^\simeq$ \cite[Corollary 8.1.3]{HebeS}. We may therefore define $\K^\add(\cC) = K(\Stab(\cC))$ as a nonconnective version of additive $K$-theory. Note that $K^\add(\cC) \to \K^\add(\cC)$ is a $T(i)$-equivalence for all $i\geq 1$.

\begin{prop}\label{prop:tstructure}
Let $\cC$ be a stable $\infty$-category. Then there is a fibre sequence $K(\cD) \to \K^\add(\cC) \to K(\cC)$ where $\cD$ is a stable $\infty$-category which admits a bounded $t$-structure. In particular, the map $K^\add(\cC) \to K(\cC)$ is 
\begin{enumerate}
	\item  a $T(i)$-local equivalence for $i\geq 2$, and
	\item  a $T(1)$-local equivalence if for all objects $X \in \cC$, the ring spectrum $\hom_\cC(X,X)[\tfrac{1}{p}]$ vanishes.
\end{enumerate}
\end{prop}
\begin{proof}
The $\infty$-category $\Fun^\add(\cC^\op,\Sp)$ admits the pointwise $t$-structure whose connective part is equivalent to $\Fun^\add(\cC^\op,\Spc)$, which is also known as $\cP_\Sigma(\cC)$ --- the nonabelian derived $\infty$-category of $\cC$ ---, and whose heart is the category $\Fun^\add(\cC^\op,\Ab)$, the Freyd envelope also studied in \cite[Ch.\ 5]{Neeman01}. When $\cC$ is stable, $\Ind(\cC)$ canonically identifies with $\Fun^\exact(\cC^\op,\Sp)$, and we obtain a localization functor
\[ \Fun^\add(\cC^\op,\Sp) \lto \Ind(\cC)\]
given by sending an additive functor to its first Goodwillie derivative.
The kernel  of this localization is generated by compact objects of the form
\begin{equation}\label{ycofiber}
	\mathrm{cofib}\left( (\mathrm{cofib}(y(A) \to y(B)) \to y(C))
\right) 
\end{equation} 
for each cofibre sequence $A \to B
\to C $ in $\mathcal{C}$; we let $\mathcal{D}$ denote the compact objects in
this kernel, or equivalently the thick subcategory generated by the above
objects \eqref{ycofiber}. In particular, there is a fibre sequence $K(\cD) \to K(\Stab(\cC)) \to K(\cC)$. We proceed by showing that $\cD$ admits a bounded $t$-structure. To that end, \cite[Theorem~4.26]{PP23} implies that the $t$-structure on $\Fun^\add(\cC^\op,\Sp)$ restricts to a $t$-structure on $\Stab(\cC)$. Moreover, we claim that $\cD$ consists precisely of the $t$-bounded objects in $\Stab(\cC)$. To see that $\cD$ consists of $t$-bounded objects, it suffices to observe that the generators \eqref{ycofiber} belong to the heart. Conversely, given a $t$-bounded object of $\Fun^\add(\cC^\op,\Sp)$, the usual expression of the
first Goodwillie derivative shows that its image in $\Ind(\cC)$ vanishes.

The theorem of the heart
\cite{Barwick} then implies that $K_{\geq 0}(\mathcal{D}) =
K_{\geq 0}(\mathcal{D}^{\heartsuit})$ admits the structure of a 
$K(\mathbb{Z})$-module and thus has vanishing $T(i)$-localizations for $i \geq
2$, cf.~\cite{Mitchell}. Thus, $L_{T(i)} K^{\mathrm{add}}(\mathcal{C}) \xrightarrow{\sim} L_{T(i)}
K(\mathcal{C})$ for $i \geq 2$ (this fact is also proved as
\Cref{additivevsexactK} below). 

Now suppose that for all objects $X \in \cC$ we have $\hom_\cC(X,X)[\tfrac{1}{p}]=0$. Then the same
holds true for objects in $\mathcal{D}$, whence $L_{T(1)}K(\mathcal{D}) = L_{T(1)}
K_{\geq 0}(\mathcal{D}^{\heartsuit})) = 0$ since $K_{\geq 0}(\mathcal{D}^{\heartsuit}) = 
K_{\geq 0}( \mathrm{Mod}_{\mathbb{F}_p}(\mathcal{D}^{\heartsuit}))$
 by the d\'evissage theorem for abelian
categories \cite[Theorem~4]{QuillenHigherK} and since  $L_{T(1)} K(\mathbb{F}_p) = 0$. Thus,  we find that $L_{T(1)}
K^{\mathrm{add}}(\mathcal{C}) \xrightarrow{\sim} L_{T(1)} K(\mathcal{C})$.  
\end{proof}

\section{Consequences and examples}

In this section we discuss some consequences and examples of our main result.

\subsection{Direct consequences}
	\label{sec:examples}

To begin with, we record some immediate corollaries of 
the Purity Theorem. 

\begin{cor} 
\label{easyvanish}
Let $R$ be a ring spectrum which is $T(n) \oplus T(n-1)$-acyclic for some $n
\geq 1$. Then $L_{T(n)}
K(R) = 0$. \qed
\end{cor} 

\begin{cor} 
Let $n \geq 2$. Then for any ring spectrum $R$, we have that the canonical map $L_{T(n)} K(\tau_{\geq
0} R ) \to L_{T(n)} K(R)$ is an equivalence. \qed
\end{cor}

In the case of $\E_\infty$-rings we furthermore find the following redshift phenomenon: 
\begin{cor}
	\label{cor:T(n)-local-vanishing-of-K(A)-for-E-infty}
Let $A$ be an $\E_{\infty}$-ring spectrum, $B$ an $A$-algebra, and let $n \geq 0$. Then $L_{T(n)} A = 0$ implies $L_{T(n+i)}K(B)= 0$ for every integer $i \geq 1$.
\end{cor}

\begin{proof}
If $A$ is $T(n)$-acyclic, then $A$, and hence also $B$, is $T(n+j)$-acyclic for all $j \geq 0$ by \cite{Hahn} and Lemma~\ref{telescope conjecture on ring spectra}. Thus, the result follows from the Purity Theorem. 
\end{proof}

Next, we include the following slight variants of \Cref{thm-A} in the
connective case, and an analog for topological cyclic homology.

\begin{cor}\label{height-1-connective}
Let $A \to B$ be a $T(1)\oplus \dots \oplus T(n)$-equivalence between connective ring spectra which induces a surjection on $\pi_0$ whose kernel is nilpotent. Then $K(A) \to K(B)$ is again a $T(1)\oplus \dots \oplus T(n)$-equivalence.
\end{cor}
\begin{proof}
Consider the pullback diagram
\[\begin{tikzcd}
	P \ar[r] \ar[d] & B \ar[d] \\
	A[\tfrac{1}{p}] \ar[r] & B[\tfrac{1}{p}]
\end{tikzcd}\]
Since $P \to A[\tfrac{1}{p}]$ is a $T(0)$-localization, applying $K$-theory to the diagram yields again a pullback, e.g.\ by \cite[Main Theorem]{LT}. Furthermore, the map $K(A[\tfrac{1}{p}]) \to K(B[\tfrac{1}{p}])$ is a $p$-adic equivalence, as $p$-adic $K$-theory is truncating on $\bbS[\tfrac{1}{p}]$-algebras \cite[Lemma 2.4]{LT} and hence also nilinvariant \cite[Corollary 3.5]{LT}. Hence, $K(P) \to K(B)$ is a $T(i)$-equivalence for all $i\geq 1$. Furthermore, $A \to P$ is a $T(0)\oplus T(1) \oplus \dots \oplus T(n)$-equivalence. \cref{purity1} together with the already established results thus implies that $K(A) \to K(B)$ is also a $T(1)\oplus \dots \oplus T(n)$-equivalence.
\end{proof}

\begin{cor}
Let $A \to B$ be a $T(1)\oplus \dots \oplus T(n)$-equivalence between connective ring spectra which induces a surjection on $\pi_0$ whose kernel is nilpotent. Then $\TC(A) \to \TC(B)$ is again a $T(1)\oplus \dots \oplus T(n)$-equivalence.
\end{cor}
\begin{proof}
By the Dundas--Goodwillie--McCarthy theorem \cite[Theorem VII.0.0.2]{DGM}, there is a cartesian square
\[
\begin{tikzcd}
K(A) \ar[r] \ar[d] & \TC(A) \ar[d] \\
K(B) \ar[r] & \TC(B).
\end{tikzcd}
\]
So we deduce the corollary from \cref{height-1-connective}.
\end{proof}

\begin{rem}
If $A \to B$ is a $T(0)$-equivalence between connective ring spectra inducing a surjection on $\pi_0$ whose kernel is nilpotent, then it is also true that the map $K(A) \to K(B)$ is a $T(0)$-equivalence. Thus, the same also holds true for $\TC(A) \to \TC(B)$.
\end{rem}

Let us call a morphism $f\colon R\to S$ of $\E_{\infty}$-ring spectra $n$-\emph{nice} if $R$ and $S$ are connective, $f$ induces an isomorphism on $\pi_0$ and after $H\Q\oplus T(1)\oplus \cdots \oplus T(n)$-localization, and $\pi_0(R) \cong \Z$. We say that $f$ is nice if it is $n$-nice for all $n\geq 0$.
\begin{cor}\label{cor:nice}
If $R\to S$ is an $n$-nice morphism of $\E_{\infty}$-ring spectra, $K(R)\to K(S)$ is again $n$-nice.
\end{cor}
\begin{proof}
By \cite[Theorem 9.53]{BGT}, we have $K_i(R) = K_i(\pi_0(R))$ for $i\leq 0$ and analogously for $S$. We deduce that $K(R)$ and $K(S)$ are connective with $\pi_0$ isomorphic to $\Z$. By \cite[Lemma 2.4]{LT}, $K(f)\colon K(R)\to K(S)$ is a rational equivalence and by \Cref{height-1-connective} a $T(1)\oplus \cdots \oplus T(n)$-equivalence.
\end{proof}

\begin{ex} The preceding corollary is tailor-made for applications to $i$-fold iterated algebraic $K$-theory
$K^{(i)}$. For example, we will argue momentarily that the canonical truncation map $K(\Z/p^k\Z) \to \Z$ is nice, so that $K^{(i)}(\Z/p^k\Z) \to K^{(i-1)}(\Z)$ is an $H\Q \oplus T(1) \oplus \dots $-equivalence for all $i\geq 1$ by induction and Corollary~\ref{cor:nice}. To see that $K(\Z/p^k\Z) \to \Z$ is nice, we first observe that the map $K_n(\Z/p^k\Z) \to K_n(\F_p)$ is an isomorphism for $n\leq 0$. This is true more generally for any quotient of a discrete ring by a nilpotent ideal, as follows from an inductive argument using the fundamental theorem of $K$-theory \cite[Ch.~IV, Corollary 8.4.1]{Weibel} and the fact that $K_0$ is invariant under quotients by a nilpotent ideal \cite[Ch.~II, Lemma 2.2]{Weibel}. Furthermore, we have $K(\Z/p^k\Z)\otimes \Q \simeq K(\F_p)\otimes \Q \simeq \Q$, using e.g.\ \cite[Corollary~5.4]{Weibel2} and Quillen's seminal calculation of the $K$-theory of finite fields \cite{Quillen}. Finally, both $K(\Z/p^k\Z)$ and $\Z$ are $T(i)$-acyclic for $i\geq 1$; for the former this follows from Corollary~\ref{cor:T(n)-local-vanishing-of-K(A)-for-E-infty} and for the latter it follows e.g.\ from Lemma~\ref{lem:basicproperties}(ii).
\end{ex}

\subsection{Examples of vanishing results}
\label{subsec:examples1}

We give various examples showing that the Purity Theorem (or
\Cref{easyvanish}) implies vanishing
statements for suitable telescopic localizations of the $K$-theory of ring
spectra; this recovers a number of existing results in the literature. 

First, we begin with the case of $K(n)$, cf.~also \cite{AKS} in the case $n =
2, p = 2,3$. 
\begin{cor}\label{K of Morava K-theory}
The spectrum $K(K(n))$ vanishes $T(m)$-locally for every $m\not= 0,n,n+1$. 
\qed
\end{cor}

\begin{rem}
\label{KofKmrem}
Using the d\'evissage result \cite[Proposition~4.4]{AGHeller} (preceded by
\cite{Barwick-Lawson} for connective $K$-theory) we can also understand
the $T(0)$-localization of $K(K(m))$. There is a fibre sequence 
\begin{equation}
	\label{eq:k(m)-K(m)-localization-sequence}
K(\F_p) \lto K(k(m)) \lto K(K(m)) 
\end{equation}
where $k(m)$ is the connective cover of $K(m)$ and the first map is induced by the functor $\Perf(\F_{p}) \to \Perf(k(m))$ given by restriction of scalars along the canonical map $k(m) \to \F_{p}$. As $K(-)\adj$ is truncating on $\bbS\adj$-acyclic ring spectra, the canonical map $K(k(m)) \to K(\F_{p})$ is a $T(0)$-equivalence.
The composite 
\[ 
K(\F_p) \lto K(k(m)) \lto K(\F_p) 
\]
is induced by the functor $\Perf(\F_p) \to \Perf(k(m)) \to \Perf(\F_p)$ sending $X$ to $X \otimes_{k(m)} \F_{p}$, which is equivalent to $\id \oplus \Sigma^{2p^m-1}$ as there is a fibre sequence
\[ 
\Sigma^{2p^m-2} k(m) \stackrel{v_m}{\lto} k(m) \lto H\F_p.
\]  
Upon applying any localizing invariant, this gives the zero map. From~\eqref{eq:k(m)-K(m)-localization-sequence} we thus obtain a fibre sequence
\[ 
K(\F_p)\adj \stackrel{0}{\lto} K(\F_p)\adj \lto K(K(m)) \adj
\]
and hence an equivalence 
\[
K(K(m))\adj \simeq K(\F_{p})\adj \oplus \Sigma K(\F_{p})\adj.
\]
\end{rem}

\begin{cor} 
\label{truncatedsphere}
For any $n \geq 0$, 
we find that $L_{T(i)} K( \tau_{\leq n}\mathbb{S}) = 0$ for $i \geq 2$. \qed
\end{cor}

Ben Antieau has already shown previously that a certain quantitative version of Proposition~\ref{prop:E(n)-local-K-theory} implies $L_{T(n)}K(\tau_{\leq m}\bbS) = 0$ at least for all $n$ such that $4p-4\geq n$, where $p$ is the implicit prime in $T(n)$, and conjectured that \Cref{truncatedsphere} is true.

\begin{cor} 
The map $K(BP\langle n \rangle) \to K(E(n))$ is a $T(i)$-equivalence for $i \geq n+1$. Furthermore, both vanish $T(i)$-locally for $i \geq n+2$. \qed
\end{cor} 

\begin{rem}
The chromatic bound for $K(BP\langle n \rangle)$ has been proved previously by Angelini-Knoll--Salch in the case where $BP\langle n \rangle$ admits an $\mathbb{E}_\infty$-structure, \cite{AKS}.

The above result implies that the sequence 
\[ K(BP\langle n-1 \rangle) \lto K(BP\langle n \rangle ) \lto K(E(n)) \]
becomes a fibre sequence after $T(i)$-localization for $i\geq n+1$. Whether or
not  this sequence is a fibre sequence (after replacing the rings with their
$p$-completions) was asked by Rognes, the $n=1$ case being a theorem of
Blumberg--Mandell \cite{BM}, and the $n=0$ case being a classical theorem of
Quillen's. It was then shown by Antieau--Barthel--Gepner that for $n\geq 2$, the
above is not a fibre sequence after rationalization, see \cite{ABG18}.  
\end{rem}

The following is an example that arose from a discussion with George Raptis.
Recall that for a connected space $X$ its Waldhausen $A$-theory is given by
$A(X) = K(\bbS[\Omega X])=K(\Sigma^{\infty}_{+}\Omega X)$. In particular,
$A(\ast) = K(\bbS)$. In the following corollary we assume that $n\geq 1$; the case $n=0$ is due to Waldhausen.
\begin{cor}
Let $W$ be a connected space. If $\Sigma^\infty W$ is $T(n) \oplus T(n-1)$-acyclic, then the
canonical maps $A(\ast) \leftrightarrows A(\Sigma W)$ are mutually inverse 
$T(n)$-equivalences. 
\end{cor}
For instance, $W$ could be a connected type $m$-complex for $m > n$.
\begin{proof}
The James splitting (see \cite[Chapter 10, Theorem 5]{MR0445484}) gives an equivalence
\[
 \Sigma^\infty \Omega \Sigma W \simeq  \Sigma^\infty \bigvee_{k \geq 1} W^{\wedge k},
\]
which implies that $\Sigma^\infty \Omega \Sigma W$ is $T(n) \oplus T(n-1)$-acyclic.  The
claim thus follows from the Purity Theorem.
\end{proof}

Next, we study the chromatic localization of the $K$-theory of certain Thom spectra $y(m)$ considered in \cite[Section 3]{MRS}; compare \cite{AKQ} for previous work in this direction. To explain the setup, we recall that for a fixed prime $p$, there is an essentially unique map of $\E_2$-spaces 
\[ 
\Omega^2 \Sigma^2 S^1 \lto \BGL_1(\bbS^\swedge_p) 
\]
sending a generator of $\pi_1$ to the element $1-p \in \pi_1(\BGL_1(\bbS^\swedge_p)) \cong \Z_p^\times$. It is a theorem of Mahowald (for $p=2$) and Hopkins (for odd primes) that its Thom spectrum is $H\F_p$ \cite{Mahowald2}; see also \cite{ACB}. We note that $\Omega^2 \Sigma^2 S^1 \simeq \Omega (\Omega S^3)$ and that $\Omega S^3$ has a canonical cell structure with one cell in every even dimension; see \cite[Corollary 17.4]{Milnor}. Let us denote by $F_m(\Omega S^3)$ the $2m$-skeleton of this cell structure. 
One then obtains maps of $\E_1$-spaces
\[ \Omega F_{p^m-1}(\Omega S^3) \lto \Omega^2 S^3 \lto \BGL_1(\bbS^\swedge_p) \]
whose Thom spectra are denoted by $y(m)$, leaving the prime $p$ implicit as always. One has $y(0) = \bbS^\swedge_p$ and $y(\infty) = H\F_p$. The above filtration of $\Omega S^3$ can also be described as the James filtration on $\Omega \Sigma S^2$, compare \cite[Section 3.1]{MRS}.

\begin{lemma}\label{y chromatic vanishing}
The spectrum $y(m)$ is $L_{m-1}^{p,f}$-acyclic. 
\end{lemma}
\begin{proof}
We need to show that $y(m)$ is $T(n)$-acyclic for $n<m$.
For $n=0$ this follows 
because $\pi_{0}(y(m)) = \F_{p}$ (see the paragraph preceding \cite[Equation 3.7]{MRS} for odd $p$ and \cite[Lemma 2.7]{AKQ} for $p=2$).
We now discuss the case where $n>0$.
Again, we distinguish the cases of even and odd primes. For $p=2$ this follows from \cite[Proposition 2.22]{AKQ} and \cref{telescope conjecture on ring spectra}. For odd primes, it is explained in \cite{MRS} that for a finite type $n$ spectrum $V_n$, the Adams spectral sequence for the spectrum $V_n \otimes y(m)$ has a vanishing line of slope $\tfrac{1}{2p^m-2}$, because this is true for $y(m)$. On the other hand, the element $v_n$ acting on $V_n$ gives an element of slope $\tfrac{1}{|v_n|}$ for the Adams spectral sequence. Hence, if $n<m$, it follows that the element $v_n$ is nilpotent on $V_n\otimes y(m)$, so that $T(n)\otimes y(m)$ vanishes as claimed.
\end{proof}

\begin{cor}\label{almost the AKQ-result}
The map $K(y(m)) \to K(\F_p)$ is an $L_{m-1}^{p,f}$-equivalence. In particular, $K(y(m))$ vanishes $T(n)$-locally for $0<n<m$.
\end{cor}
\begin{proof}
The vanishing follows immediately from the Purity Theorem and \cref{y chromatic vanishing}. 
The map is also a $T(0)$-equivalence, as $T(0)$-local $K$-theory is truncating on $T(0)$-acyclic ring spectra by a result of Waldhausen (see also \cite[Lemma 2.4]{LT}).
\end{proof}

\begin{rem}
This corollary implies the corresponding statement with the $T(i)$ replaced by the Morava $K$-theories $K(i)$. For $p=2$, the latter was previously studied by Angelini-Knoll and Quigley \cite[Theorem 1.3]{AKQ} using trace methods.
\end{rem}

We obtain a similar result for the integral versions $z(m)$ of $y(m)$ which appear in \cite{AKQ} when $p=2$. Again, there are versions for odd primes, but we refrain from spelling them out here.
\begin{cor}
The map $K(z(m)) \to K(\Z_{(2)})$ is a $T(n)$-equivalence for $0<n<m$. 
\end{cor}
\begin{proof}
By \cite[Proposition 2.22]{AKQ}, $z(m)$ is $K(n)$-acyclic for $1\leq n< m$, and hence also $T(n)$-acyclic for $1 \leq n <m$, again by \cref{telescope conjecture on ring spectra}. The corollary then follows from \cref{height-1-connective}.
\end{proof}

\subsection{Examples of purity}
We list some further examples of purity statements, special cases of which have been studied in the literature before.

\begin{cor}
	\label{cor:K-of-L1p-localization-for-ko-algs}
Let $A$ be a $ko$-algebra. Then the natural map $K(A) \to K(A[\frac{1}{\beta}])$ is a $T(n)$-local equivalence for all $n \geq 2$. 
\end{cor}
\begin{proof}
This follows immediately from the Purity Theorem, as the map 
$ko \to ko[\beta^{-1}]=KO$ is a $T(n)$-equivalence for all $n\geq 1$, hence so is
$A \to A[\frac{1}{\beta}]$ for every $ko$-algebra.
\end{proof}

\begin{rem}\label{rem:Blumberg-Mandell-sequence} 
By work of Blumberg--Mandell \cite{BM}, there is a fibre sequence of connective $K$-theory spectra
\[ K^{\mathrm{cn}}(\Z) \lto K^{\mathrm{cn}}(\ku) \lto K^{\mathrm{cn}}(KU) \]
and likewise for $\ko$ and $KO$ in place of $\ku$ and $KU$. 
Together with Mitchell's result this implies \cref{cor:K-of-L1p-localization-for-ko-algs} in the case where $A$ is $\ko$ or $\ku$. Note also that for $ko$-algebras $A$, we have that $K(A)$ is $T(n)$-acyclic for $n\geq 3$; this follows from \cref{easyvanish}, but was shown for $A = ku$ and $p\geq 5$ already in \cite{AR} and in general in \cite{CMNN2}. Thus \cref{cor:K-of-L1p-localization-for-ko-algs} is a useful statement only at height 2.
\end{rem}

We get a similar result for algebras over the connective spectrum of topological modular forms $\tmf$, see
\cite{TMF,Behrens} for introductions. Recall that $\tmf$ is by definition the connective cover of an $\E_{\infty}$-ring spectrum $\Tmf$ that arises as the global sections of a sheaf $\cO^{top}$ of $\E_{\infty}$-ring spectra on the \'etale site of the compactified moduli stack of elliptic curves $\overline{\cM}_{ell}$. The evaluation of $\cO^{top}$ on the uncompactified moduli stack $\cM_{ell}$ is the periodic spectrum $TMF$. Of the following corollary, the first statement was already proven in \cite{CMNN2}.

\begin{cor}\label{cor:tmf-algebras}
Let $A$ be a $\tmf$-algebra. 
\begin{enumerate}
	\item The spectrum $K(A)$ vanishes $T(n)$-locally for all $n\geq 4$. 
	\item The map $K(A) \to K(A\tensor_{\tmf}\Tmf)$ is a $T(n)$-equivalence for all $n\geq 2$. 	
	\item The map $K(A) \to K(A\tensor_{\tmf}TMF)$ is a $T(3)$-equivalence.
	\item At the prime $2$, there is a $T(3)$-local equivalence $K(\tmf) \simeq K(TMF) \simeq K(E_2)^{hGL_2(\F_3)}$, where $E_2$ denotes the Lubin--Tate spectrum for a supersingular elliptic curve over $\F_4$. Replacing $GL_2(\F_3)$ by the group of automorphisms over $\mathbb{F}_3$ of a supersingular elliptic curve over $\mathbb{F}_9$, the analogous statement holds at the prime $3$ as well.
\end{enumerate}
\end{cor}
\begin{proof}
By \cite[Theorem 2.1]{Ravenel84}, the spectrum $BP[v_n^{-1}]$ has the same Bousfield class as $K(0)\oplus \cdots \oplus K(n)$ and is thus $L_n$-local. Thus, every $p$-local complex oriented ring spectrum whose formal group law has height at most $n$ is $L_n$-local. Evaluated on any affine, $\cO^{top}$ is even and hence complex orientable; moreover its formal group is isomorphic to that of the corresponding generalized elliptic curve and thus has height at most $2$ at any prime. We see that $\Tmf_{(p)}$ is, as a limit of $L_2$-local spectra, itself $L_2$-local and thus $T(n)$-acyclic for $n\geq 3$.  As $\tmf \to \Tmf$ is a $T(n)$-equivalence for all $n\geq 1$ by \cref{lem:basicproperties}, we can deduce moreover that $\tmf$ is $T(n)$-acyclic for all $n \geq 3$. Thus the first two statements follow from our main theorem. 
		
For the third statement, it suffices to show that $\Tmf \to TMF$ is a $T(n)$-equivalence for $n = 2$ (and hence all $n\geq 2$). As taking global sections of quasi-coherent $\cO^{top}$-modules preserves colimits by \cite{Mathew-Meier}, we have $T(2) \tensor Tmf \simeq \Gamma(T(2)\tensor \cO^{top}$); hence it suffices to show that $T(2)\tensor \cO^{top}(\Spec A) \to T(2)\tensor \cO^{top}(\Spec A)[\Delta^{-1}]$ is an equivalence for every \'etale affine $\Spec A \to \overline{\cM}_{ell}$, where $\Delta$ denotes the discriminant. As all generalized elliptic curves of height $2$ are actually smooth elliptic curves, inverting $v_2$ (as we do in $T(2)$) indeed inverts $\Delta$ as well.
		
Note that  $K(n)$-localization and $T(n)$-localization coincide on $L_n$-local spectra. Indeed, if $X$ is $L_{n}$-local, then the fibre of the map $X \to L_{K(n)}X$ is $L_{n}$-local and $K(n)$-acyclic, whence $L_{n-1}$-local and thus $T(n)$-acyclic. As $L_{K(n)}X$ is also $T(n)$-local, it follows that $L_{T(n)}X \simeq L_{K(n)}X$.
 Thus, $\tmf \to TMF \to L_{K(2)}TMF$ are $T(2)$-local equivalences and hence induce $T(3)$-equivalences in $K$-theory by our main theorem. The faithful $GL_2(\F_3)$-Galois extension $TMF_{(2)} \to TMF(3)_{(2)}$ from \cite[Theorem 7.6]{Mathew-Meier} localizes to the Galois extension $L_{K(2)}TMF \to L_{K(2)}TMF(3)\simeq E_2$ (cf.\ \cite[Proposition 6.6.10]{Behrens}, \cite[Proposition 3.6]{Heard--Mathew--Stojanoska}). Thus, the map $K(L_{K(2)}TMF) \to K(E_2)^{hGL_2(\F_3)}$ is an equivalence after an arbitrary telescopic localization by \cite[Theorems 5.6, Corollary B.4]{CMNN}. The statement for $p=3$ is proven analogously using that here $L_{K(2)}TMF \simeq E_2^{hG_{24}}$, where $E_2$ is the Lubin--Tate spectrum for a supersingular elliptic curve $C$ over $\F_9$ and $G_{24}$ is its group of automorphisms over $\mathbb{F}_3$.
\end{proof}

\begin{rem}
In \cite{Barwick-Lawson} Barwick and Lawson provide an analog of the Blumberg--Mandell localization sequence (see Remark~\ref{rem:Blumberg-Mandell-sequence}) for certain regular ring spectra. In particular, there is a localization sequence of connective $K$-theory spectra
\[
K^{\mathrm{cn}}(\Z) \lto K^{\mathrm{cn}}(\tmf) \lto K^{\mathrm{cn}}(\Tmf),
\]
which implies the second part of the previous corollary for $A = \tmf$ and certain other regular $\tmf$-algebras.
We also remark that in \cite{AGHeller} these localization sequences are extended to non-connective $K$-theory spectra. However, for the present application this is irrelevant as the difference vanishes after telescopic localization.
\end{rem}

\subsection{\texorpdfstring{Consequences for \textit{K}(1)-local \textit{K}-theory}
{Consequences for K(1)-local K-theory}}

\label{subsec:K1-local-consequences}

We now record the consequences of 
the Purity Theorem at height $1$ (or, equivalently, 
\Cref{thm-A}). 
Recall also that $K(1)$ and $T(1)$-localization coincide.

\begin{cor}
	\label{cor:K1local-truncating} 
$K(1)$-local $K$-theory is truncating on $K(1)$-acyclic ring spectra.
In fact, for a $K(1)$-acyclic ring spectrum $A$, we have 
$L_{K(1)} K(A) = L_{K(1)} K(A[\tfrac{1}{p}])$. 
\end{cor}
\begin{proof}
Let $A$ be a $K(1)$-acyclic, connective ring spectrum. By \cref{purity1} we have an equivalence $L_{K(1)}K(A) \simeq L_{K(1)}K(A\adj)$. The claim follows from this as $p$-adic $K$-theory is truncating on $\bbS/p$-acyclic ring spectra by a result of Waldhausen (see also \cite[Lemma~2.4]{LT}).
\end{proof}

In the case of $H\mathbb{Z}$-algebras, the last assertion of 
\Cref{compKTC}
also appears in \cite{BCM}, proved by different methods. 
From \cite[Theorems~3.3, A.2]{LT} we then get the following. Note that discrete rings are $K(1)$-acyclic.
\begin{cor}
	\label{K1localKcdh}
$K(1)$-local $K$-theory of discrete rings is nilinvariant and satisfies Milnor excision and cdh-descent.
\end{cor}

We also get the following consequence. 
\begin{cor}
	\label{K1localKhomotopy}
Let $A$ be a connective and $K(1)$-acyclic ring spectrum. Then the canonical map $L_{K(1)}K(A) \to L_{K(1)}K(A[x])$ is an equivalence. In other words, $K(1)$-local $K$-theory is homotopy invariant on connective, $K(1)$-acyclic ring spectra.
\end{cor}
Here, for any ring spectrum $A$, the symbol $A[x]$ denotes the ring spectrum $A \otimes \Sigma^{\infty}_{+}\Z_{\geq 0}$.
\begin{proof}
We observe that $A[x] = A \otimes \bbS[x]$ is also $K(1)$-acyclic and connective. Hence, by \cref{cor:K1local-truncating} we may assume that $A$ is discrete so that $A[x]$ is the usual discrete polynomial ring $A\otimes_\Z \Z[x]$. By the same corollary we may furthermore assume that $p$ is invertible in $A$. In this case, Weibel \cite{Weibel2} has shown that $p$ is also invertible on $NK(A) = \fib(K(A[x]) \to K(A))$. So the $p$-completion of $NK(A)$ vanishes, and hence $L_{K(1)}NK(A) = 0$ as well.
\end{proof}

\begin{rem}
For ring spectra, there are two canonical ``affine lines:'' The flat affine line $A[x]$ as above, and the smooth affine line $A \otimes \bbS\{x\}$, where $\bbS\{x\}$ is the free $\E_\infty$-ring on a degree $0$ generator. Since the canonical map $\bbS\{x\} \to \bbS[x]$ is a $\pi_0$ isomorphism, we also obtain homotopy invariance with respect to the smooth affine line on connective, $K(1)$-acyclic ring spectra $A$: Both maps $K(A) \to K(A\{x\})$ and $K(A\{x\}) \to K(A[x])$ are $K(1)$-local equivalences.
\end{rem}

\begin{rem}
Recall that Weibel's homotopy $K$-theory $KH(A)$ of a discrete ring $A$ is defined as the geometric realization of the simplicial spectrum $K(A[\Delta^{\bullet}])$ with 
\[A[\Delta^{n}] = A[x_{0}, \dots, x_{n}]/(x_{0}+\dots+x_{n}-1) \cong A[x_{1}, \dots, x_{n}].\]
 It follow from the above corollary that
\[
L_{K(1)}K(A) \simeq L_{K(1)}KH(A).
\]
By results of Weibel and Cisinski homotopy $K$-theory satisfies Milnor excision \cite{Weibel2} and cdh-descent \cite{Cisinski}. In this way we obtain another proof of Corollary~\ref{K1localKcdh}.
\end{rem}

\begin{rem}\label{TC-version-of-question}
The analog of Corollary~\ref{cor:K1local-truncating} for topological cyclic homology does not hold:
As $\THH(\Z[\frac{1}{p}])$ is a $\Z[\frac{1}{p}]$-algebra, it vanishes $p$-adically. So $\TC(\Z[\frac{1}{p}])$ vanishes $p$-adically, and a fortiori after $T(1)$-localization. However, $L_{T(1)}\TC(\Z)$ does not vanish: For odd primes $p$,
 B\"okstedt and Madsen \cite{MR1317117} computed the connective cover of $\TC(\Z)_{p}^{\swedge} \simeq \TC(\Z_{p})_{p}^{\swedge}$ to be equivalent to $j \oplus \Sigma j \oplus \Sigma^{3} ku_{p}^{\swedge}$ where $j$ is the connective cover of the $K(1)$-local sphere. In particular, the $T(1)$- or equivalently $K(1)$-localization of $\TC(\Z)$ is given by
\begin{equation}\label{eq:TCZ}
L_{K(1)} \TC(\Z) \simeq L_{K(1)}\bbS \oplus \Sigma L_{K(1)} \bbS \oplus \Sigma^{3} KU_{p}^{\swedge} \not= 0.
\end{equation}
For $p=2$, \cite[Theorem 0.5]{Rognes} and \cite[Formula (0.2)]{Rognes2} give a filtration of $L_{K(1)}\TC(\Z)$, whose associated graded essentially looks like the summands in \eqref{eq:TCZ}. Using that $KU_{p}^{\swedge}$ is rationally non-trivial in infinitely many degrees, while the other two terms are rationally non-trivial only in finitely many degrees, one obtains that $L_{K(1)}\TC(\Z)$ is non-trivial at $p=2$ as well.
\end{rem}

\begin{rem} 
We point out that, although $K(1)$-local $\TC$ commutes with filtered colimits of rings \cite[Theorem G]{CMM}, it does not commute with filtered colimits of categories. 
In fact, one checks that (cf.~also \cite[Proposition~2.15]{BCM})
the filtered colimit $\colim\limits_k \Mod_{\mathbb{Z}/p^k}( \Perf(\mathbb{Z}))$
is the $\infty$-category of $p$-power torsion perfect $\mathbb{Z}$-modules and we thus obtain an exact sequence 
\[ \colim\limits_k \Mod_{\Z/(p^k)}(\Perf(\Z)) \lto \Perf(\Z) \lto \Perf(\Z\adj).\]
Assuming that $L_{K(1)}\TC$ commutes with filtered colimits we find that
$L_{K(1)}\TC$ of the fibre vanishes, as
$L_{K(1)}\TC(\Mod_{\Z/(p^k)}(\Perf(\Z)))$ is a module over
$L_{K(1)}\TC(\Z/(p^k))$ which vanishes since $L_{K(1)} K(\mathbb{Z}/p^k) = 0$ as
above;
this is a contradiction.
\end{rem}

\subsection{Consequences for $T(n)$-local $K$-theory for $n\geq 2$}
\label{subsec:KTC}

In this subsection, we record some further structural 
consequences of 
the Purity Theorem 
at heights $\geq 2$. 
Some further structural features in this context are
also explored in \cite{CMNN2}. 

\begin{cor} 
\label{compKTC}
Let $n \geq 2$. Then for any ring spectrum $A$, we have 
a natural equivalence
$L_{T(n)} K(A) = L_{T(n)} K( \tau_{\geq 0} A )  =  L_{T(n)} \mathrm{TC}( \tau_{\geq 0} A)$. 
\end{cor} 
\begin{proof} 
Indeed, this follows because $\tau_{\geq 0} A \to A$ 
induces an equivalence on $L_{T(n)} K(-)$ by \Cref{purity1}. 
Now we use the Dundas--Goodwillie--McCarthy theorem \cite{DGM} combined with Mitchell's
theorem \cite{Mitchell} to obtain $L_{T(n)} K( \tau_{\geq 0} A) \xrightarrow{\sim} L_{T(n)}
(\mathrm{TC}(\tau_{\geq 0}  A ))$, hence the result. 
\end{proof} 

For a $T(n)$-acyclic ring spectrum (with $n\geq 2$), we therefore obtain an equivalence
\[ L_{T(n)} K(L_{T(n-1)}A) \simeq L_{T(n)} \TC(\tau_{\geq 0}A).\]
This should be compared to the result obtained in \cite{BCM} that if $A$ is a
commutative, $p$-adically complete ring, then 
\[ L_{T(1)} K(A[\tfrac{1}{p}]) \simeq L_{T(1)} \mathrm{TC}(A).  \]
Note by contrast that no commutativity or completeness at $(p,v_1,\dots,v_{n-1})$ is required in \Cref{compKTC}. 

Next, we record a result describing the behavior of group-complete $K$-theory
versus 
Waldhausen $K$-theory; this is essentially a restatement of the above in
categorical terms, and informally states that 
for $T(i)$-local phenomena with $i \geq 2$, it suffices simply to 
group-complete (rather than split all cofibre sequences) in the definition of
$K$-theory. 
 Given an additive $\infty$-category $\mathcal{A}$, we let
$K^{\mathrm{add}}(\mathcal{A})$ 
denote the group-completion $K$-theory of $\mathcal{A}$, which we regard as a
connective spectrum. For another proof of this result,
cf.~\Cref{remark:tstructure}. 

\begin{cor} 
\label{additivevsexactK}
Let $\mathcal{C}$ be a stable $\infty$-category, and let $\mathcal{A} \subset
\mathcal{C}$ be an  additive subcategory. Suppose $\mathcal{A}$ generates
$\mathcal{C}$ as a thick subcategory. 
Then the natural map $K^{\mathrm{add}}(\mathcal{A}) \to K(\mathcal{C})$ induces
an equivalence on $T(i)$-localization, for $i \geq 2$. 
\end{cor} 
\begin{proof} 
By passage to filtered colimits, we can assume that $\mathcal{A}$ is generated
under coproducts by a single object $X$ (and hence $\mathcal{C}$ is
generated as a thick subcategory by $X$). 
In particular, we have an equivalence $\tau_{\geq
1}K^{\mathrm{add}}(\mathcal{A})
\simeq \tau_{\geq 1}
K( \tau_{\geq 0} \mathrm{End}_{\mathcal{C}}(X))$, while $K(\mathcal{C}) = K(
\mathrm{End}_{\mathcal{C}}(X))$. 
The result then follows from \Cref{compKTC}. 
\end{proof} 

\begin{cor}
Let $n \geq 2$. 
The construction $A \mapsto L_{T(n)} K(A)$, from ring spectra to $T(n)$-local spectra, preserves sifted colimits. 
The same holds if we restrict to the subcategory of $T(n-1)$-local ring spectra. 
\end{cor}
\begin{proof} 
We use here that the construction 
$R \mapsto \mathrm{TC}(R)/p$, from connective ring spectra to spectra,
preserves sifted colimits, cf.~\cite[Corollary~2.15]{CMM}. 
We prove the first claim that $A \mapsto L_{T(n)} K(A)$ preserves sifted
colimits as $A$ ranges over all ring spectra. 
Let $A_i, i \in I$ be a sifted diagram of ring spectra. 
Then $\tau_{\geq 0} A_i, i \in I$ yields a sifted diagram of connective
ring spectra, and using \cref{lem:basicproperties} we find that 
\[ 
L_{T(n)}\big(
\colim_{i \in I} \mathrm{TC}({\tau_{\geq 0}}A_i) \big)
\xrightarrow{\sim} 
L_{T(n)} \mathrm{TC}( \colim_{i\in I} \tau_{\geq 0} A_i)  . \]
Using that $\colim_{i} \tau_{\geq 0} A_i \to \tau_{\geq 0}
(\colim_i A_i)$ is a $T(1) \oplus \dots \oplus T(n)$-equivalence and hence induces an
equivalence on $L_{T(n)} K(-)$ thanks to \Cref{purity1}, 
we conclude the result from \Cref{compKTC}.

Finally, suppose $A_i$ is  a sifted diagram of $T(n-1)$-local ring spectra. 
Then the map 
$\mathrm{colim}_{i }A_i \to L_{T(n-1)}( \mathrm{colim} A_i)$ is a $T(n-1) \oplus
T(n)$-equivalence (as $T(n-1)$-local spectra are $T(n)$-acyclic), hence the last claim follows by what has already been
proved and the Purity Theorem.
\end{proof}

Finally, we record the $T(n)$-local (for $n \geq 2$) analog of the Farrell--Jones
conjecture; the following has been also
observed by Marco Varisco for connective ring spectra. We refer to the surveys \cite{RV, Lueck} for an introduction to this conjecture and its applications.
This will rely on the following result about the 
assembly map in $p$-adically completed topological cyclic homology, which
follows by combining  a result
of L\"uck--Reich--Rognes--Varisco \cite{LRRV} and finiteness properties of
$\mathrm{TC}$ from \cite{CMM}. 
By contrast, the assembly map from (non-$p$-completed)
$p$-typical $\mathrm{TC}$ need not be an equivalence for the family of cyclic
subgroups, cf.~\cite[Sec.~6]{LRRV}.

\begin{prop} 
\label{TCpassembly}
Let $R$ be any connective ring spectrum, and let $G$ be any group. 
Let $\mathcal{O}_{\mathscr{Cyc}}(G)$ be the subcategory of the orbit category of
$G$ spanned by $G$-sets of the form $G/H$, with $H \subset G$ cyclic. 
Then the assembly map 
\begin{equation*} \label{assemblymap1}  \colim\limits_{G/H \in
\mathcal{O}_{\mathscr{Cyc}}(G)} \mathrm{TC}( R[H])
\to \mathrm{TC}( R[G])   \end{equation*}
is a $p$-adic equivalence. 
\end{prop} 
\begin{proof} 
By \cite[Theorem 1.19]{LRRV}, the assembly map
for the family of cyclic groups 
for $\mathrm{THH}$ 
is an equivalence. Since $\mathrm{TC}/p$ commutes with colimits as a functor
from connective cyclotomic spectra to spectra, \cite[Theorem 2.7]{CMM}, the
result follows.
\end{proof}

\begin{cor} 
Let $R$ be any ring spectrum, and let $G$ be any group, and let
$\mathcal{O}_{\mathscr{C}}(G)$ be as in \Cref{TCpassembly}. 
Then the assembly map 
\begin{equation*} \label{assemblymap}  \colim\limits_{G/H \in \mathcal{O}_{\mathscr{Cyc}}(G)} K( R[H])
\to K( R[G])   \end{equation*}
is a $T(n)$-equivalence for $n \geq 2$. 
\end{cor} 
\begin{proof} 
By \Cref{compKTC}, we may assume $R$ is connective, and replace $K$ by
$\mathrm{TC}$.
The result follows from \Cref{TCpassembly}. 
\end{proof}

\bibliographystyle{amsalpha}
\bibliography{excision}

\end{document}